\documentclass[a4paper,11pt,fleqn]{article}
\usepackage[english]{babel}
\usepackage[T1]{fontenc}
\usepackage[ruled,vlined]{algorithm2e}
\usepackage{amsmath,amsfonts,amsthm}
\usepackage{thm-restate,thmtools}
\usepackage{algorithmic}
\usepackage{nicefrac}
\usepackage[utf8]{inputenc}
\usepackage{color}
\usepackage{xcolor,colortbl}
\usepackage{bm,bbm}
\usepackage{soul}
\usepackage{mathtools}
\usepackage{nicefrac}  
\usepackage{booktabs}
\usepackage{adjustbox}
\usepackage{hyperref}
\usepackage{url}
\usepackage{cleveref}
\usepackage[flushleft]{threeparttable}
\usepackage{multirow}
\usepackage[shortlabels]{enumitem}
\usepackage{pifont}
\usepackage{microtype}      
\usepackage{graphicx}
\usepackage{doi}
\usepackage{authblk}
\usepackage[toc,page]{appendix}

\usepackage{amsmath,amssymb,amsthm}
\usepackage{mathrsfs} 
\usepackage{euscript}
\usepackage{charter}
\usepackage{array}
\usepackage[normalem]{ulem}

\definecolor{labelkey}{rgb}{0,0.08,0.45}
\definecolor{refkey}{rgb}{0,0.6,0.0}
\definecolor{Brown}{rgb}{0.45,0.0,0.05}
\definecolor{dgreen}{rgb}{0.00,0.49,0.00}
\definecolor{dblue}{rgb}{0,0.08,0.75}

\hypersetup{
colorlinks,
linkcolor=dgreen,
citecolor=dblue
}

\numberwithin{equation}{section}

\newtheorem{lemma}{Lemma}
\newtheorem{proposition}{Proposition}

\newtheorem{corollary}{Corollary}
\newtheorem{theorem}{Theorem}

\theoremstyle{definition}

\newtheorem{remark}{Remark}

\providecommand{\norm}[1]{\lVert#1\rVert}
\providecommand{\scalarp}[1]{\langle#1\rangle}
\providecommand{\abs}[1]{\lvert#1\rvert}

\DeclareMathOperator*{\argmin}{\text{\rm{argmin}}}
\newcommand{\minimize}[2]{\ensuremath{\underset{\substack{{#1}}}{\text{\textrm{minimize}}}\;\;#2 }}

\newcommand{\EE}{\ensuremath{\mathsf E}}

\newcommand{\R}{\ensuremath \mathbb{R}}
\newcommand{\HH}{\ensuremath \mathsf{H}}

\newcommand{\XX}{\ensuremath \mathsf{X}}

\newcommand{\xx}{\ensuremath \mathsf{x}}

\newcommand{\zz}{\ensuremath \mathsf{z}}

\newcommand{\uu}{\ensuremath \mathsf{u}}
\newcommand{\yy}{\ensuremath \mathsf{y}}

\newcommand{\N}{\ensuremath \mathbb{N}}

\newcommand{\Le}{L_\varepsilon}

\newcommand{\bbE}{\mathbb{E}}
\newcommand{\bbI}{\mathbb{I}}
\newcommand{\bbN}{\mathbb{N}}
\newcommand{\bbP}{\mathbb{P}}
\newcommand{\bbR}{\mathbb{R}}

\newcommand{\cC}{\mathcal{C}}
\newcommand{\cD}{\mathcal{D}}
\newcommand{\cE}{\mathcal{E}}
\newcommand{\cF}{\mathcal{F}}

\newcommand{\cL}{\mathcal{L}}

\newcommand{\cX}{\mathcal{X}}
\newcommand{\cY}{\mathcal{Y}}
\newcommand{\cZ}{\mathcal{Z}}

\newcommand{\ug}{u}

\title{ { \sffamily An Improved Analysis of the Clipped Stochastic subGradient Method under Heavy-Tailed Noise} } 

\author[1]{Daniela A. Parletta}
\author[2]{Andrea Paudice}
\author[4, 5]{Saverio Salzo}

\affil[1]{\footnotesize Department of Mathematics, University of Genoa, Via Dodecaneso 35, 16146 Genova, Italy}
\affil[2]{\footnotesize Department of Computer Science, Aarhus University, Åbogade 34, Aarhus, 8200,  Denmark}
\affil[4]{\footnotesize DIAG, Sapienza University of Rome, Via Ariosto, 25, 00185 Roma, Italy}
\affil[5]{\footnotesize CSML, Istituto Italiano di Tecnologia, Via Enrico Melen 83, 16152 Genova, Italy}

\date{}

\begin{document}
\maketitle

\begin{abstract}
In this paper, we provide novel optimal (or near optimal) convergence rates for a clipped version of the stochastic subgradient method. We consider nonsmooth convex problems over possibly unbounded domains, under heavy-tailed noise that possesses only the first $p$ moments for $p \in \left]1,2\right]$. For the last iterate, we establish convergence in expectation for the objective values with rates of order $(\log^{1/p} k)/k^{(p-1)/p}$ and $1/k^{(p-1)/p}$, for anytime and finite-horizon respectively. We also derive new convergence rates, in expectation and with high probability, for the objective values along the average iterates---improving existing results by a $\log^{(2p-1)/p} k$ factor. Those results are applied to the problem of supervised learning with kernels demonstrating the effectiveness of our theory. Finally, we give preliminary experiments.
\end{abstract}

\vspace{1ex}

\allowdisplaybreaks

\section{Introduction}
The \emph{Stochastic subGradient Method} (SsGM) proposed for the first time in \cite{Ermoliev1969}, is arguably one of the most popular methods in modern machine learning due to its simplicity and effectiveness. Indeed, it is widely used in empirical (regularized) risk minimization on large-scale datasets \cite{Bottou2018}, and in supervised learning \cite{Shalev2009,Shalev2010,Shalev2011}. For nonsmooth Lipschitz convex problems, which are the focus of this paper, the convergence of SsGM (and that of its non-Euclidean generalization \emph{Stochastic Mirror Descent} (SMD)) has been extensively investigated in \cite{Nemirovski1983,Nemirovski2009,Shamir2013,Harvey2019a,Jain2021,Liu2023b,Bach2024,Eldowa2024}. These results concern either the average of the produced iterates or the last iterate, with the latter being the preferred choice by the practitioners due to its simplicity and often superior empirical performances. 

The standard analysis of the convergence in expectation assumes the noise to have uniformly bounded variance \cite{Nemirovski2009,Shamir2013,Jain2021,Bach2024} and shows rates of order  $1/\sqrt{k}$ (possibly up to log factors) after $k$ iterations, which is optimal \cite{Nemirovski1983}. However, recent works have questioned the validity of the uniformly bounded variance assumption in modern datasets \cite{Zhang2020,Zhou2020,Gurbuzbalaban2021a,Gurbuzbalaban2021b}, and proposed a more general model where only the first $p \in \left]1,2\right]$ moments are bounded (with the special case $p=2$ corresponding to the canonical bounded variance setting). In this regime, called \emph{heavy-tails}, the work \cite{Zhang2020} shows that the SsGM cannot converge on quadratic problems and similar results are conjectured to hold also for nonsmooth problems. \emph{Clipping} offers a workaround that preserves the simplicity and the efficiency of SsGM, while increasing its robustness \cite{Nazin2019,Zhang2020,Gorbunov2020,Gorbunov2021,Jakovetic2023,Liu2023b,Nguyen2023,Sadiev2023}. Indeed, \cite{Liu2023b} showed that, relying on a clipping, rates of order  $1/k^{(p-1)/p}$ (up to polylog factors) are possible. We notice that, while being nearly optimal \cite{Vural2022}, these results only hold for the average iterate.

\subsection{Main contribution and comparison with previous works}
Motivated by the above discussion, we provide the following contributions.
\begin{itemize}
\item We derive the first convergence rates in expectation  for the objective values at \emph{the last iterate} of clipped SsGM under a uniformly bounded $p$-th moment condition. In particular, we show that in the \emph{finite-horizon} setting---when the horizon $k$ is known and can be used to set the algorithm's parameters---the last iterate achieves an optimal convergence rate in expectation of order  $1/k^{(p-1)/p}$. See \Cref{thm:li_unbounded_fh}. On the other hand, in the \emph{anytime} setting (aka \emph{infinite-horizon}, i.e., when the horizon is not finite), we show that the corresponding rate is of order $(\log^{1/p} k)/k^{(p-1)/p}$ (ignoring polyloglog factors). See \Cref{thm:li_unbounded_at}.
Our analysis is based on a reduction to the error of the average iterate, which also shares the same convergence rate (\Cref{crl:avg_unbounded}), improving upon the state-of-the-art bound by shaving off a 
$\log^{(2p-1)/p)} k$ factor.

\item We prove rates of convergence in high probability for the objective values along the average iterates that
match the corresponding ones in expectation up to a $\log(1/\delta)$ factor, accounting for the confidence level. These rates improve the state of the art reducing the rate of convergence from a $(\log^2\! k)/k^{(p-1)/p}$ to $(\log^{1/p}\! k)/k^{(p-1)/p}$. See \Cref{thm:avg_unbounded_hp} and \Cref{crl:avg_unbounded_hp}.

\item From a technical perspective, the clipping opearation leads to a biased estimator of the subgradient which introduces the challenge of handling an additional term in the recursive inequality that governs the behavior of the iterates. This issue can be easily resolved only when the constraint set is bounded (see, for example, \cite{Parletta2024}). We address this problem in the general case by presenting an unrolling result for the associated recurrence (see \Cref{lm:recurrences}).

\item Our analysis accommodates a wider range of parameter settings compared to previous works, which proves to be more effective in practical applications.

\item Finally, we discuss the implications of our results for the case of supervised learning with kernels, showing that clipped SsGM can be easily \emph{kernelized} and achieves optimal performance under substantially more general assumptions on the data distribution than traditional methods.
\end{itemize}


\paragraph{\upshape Related Work} 
Under the uniformly bounded variance assumption, the analysis of the average iterate of SsGM is straightforward and yields rates of order  $1/\sqrt{k}$ and $(\log k)/\sqrt{k}$ in the finite horizon and anytime settings, respectively \cite{Nemirovski1983}. The last iterate matches this performance, but its analysis is more involved and depends on whether $k$ is known \cite{Jain2021} or not \cite{Shamir2013}. 
Notably, the latter works analyzes the last iterate behavior under the strong assumption that the domain is bounded.

Under the more general bounded $p$-th moment setting \cite{Zhang2020}, existing works consider a clipped version of SsGM \cite{Liu2023b,Nguyen2023,Sadiev2023}. The work \cite{Nguyen2023} improves upon \cite{Sadiev2023} and, for smooth convex problems, shows that the average iterate features high probability rates of order $1/k^{(p-1)/p}$ and $(\log^{2/p} k)/k^{(p-1)/p}$ in the finite horizon and anytime settings, respectively. The work most closely related to ours is \cite{Liu2023b} which considers nonsmooth convex problems, although it focuses solely on the average iterate. The authors show convergence in expectation with rates of order $1/k^{(p-1)/p}$ and $(\log^2 k)/k^{(p-1)/p}$ for known and unknown $k$, respectively. In the same work, they also established rates of convergence in high probability with a confidence overhead of order $\log(1/\delta)$. Note that $1/k^{(p-1)/p}$ is optimal in this regime \cite{Vural2022}. Our results on the average iterate improve upon the state of the art by reducing the numerator to $\log^{1/p} k \log^2\log k$ in the anytime setting. Moreover, unlike \cite{Liu2023b}, we allow for a parameter setting in which the initial clipping level is chosen as close as needed to the Lipschitz constant of the objective function---a feature that can be beneficial in practical applications. In addition, we extend the convergence rates in expectation to the last iterate. We summarize this discussion in Table \ref{tab:results}.

A related line of research \cite{Vural2022,Liu2024} considers SMD with \emph{uniformly convex regularizers}. These works show rates of the same order as ours, although the analysis of the average iterate is restricted to bounded domains and finite horizon \cite{Vural2022}. One shortcoming of this method is that it is unclear whether the updates can be concretely implemented, as the closed-form solution to the Bregman projection problem (induced by that specific regularizer) may be unavailable. Indeed, the authors leave open any discussion on the computational tractability of this method. 

\paragraph{\upshape Organization} 
The rest of the paper is organized as follows. In \Cref{sec:2}, we present the assumptions, the algorithm, and the different schedules for the related parameters. In \Cref{sec:AV}, we carry out the convergence analysis for the average iterates of C-SsGM, providing rates both in expectation and in high probability. In \Cref{sec:LI}, we study the convergence rates for the last iterate of C-SsGM in both finite horizon and anytime settings. \Cref{sec:SL} addresses the application of the proposed algorithm and results to the problem of supervised learning with kernels. Sections~\ref{sec:experiments} and \ref{sec:conclusion} contain preliminary numerical experiments and conclusions, respectively. Finally, Appendices~\ref{app:A}, \ref{app:B},
and \ref{sec:appC} include additional technical results required for the convergence analysis, as well as implementation details for the application to supervised learning problems.

\begin{table}[t]
\centering
\resizebox{0.90\textwidth}{!}{%
\begin{tabular}{lc|c||c|c}
\multirow{2}{*}{} &
\multicolumn{2}{c}{SOTA} &
\multicolumn{2}{c}{This Work} \\
\midrule
\hline
& AT & FH & AT & FH \\ 
\hline
 & & & & \\
\textsc{Last Iterate} & $\dfrac{\log k}{\sqrt{k}}$ (\text{only } $p=2$) & $\dfrac{1}{\sqrt{k}}$ (\text{only } $p=2$) & $\boxed{\dfrac{\log^{1/p} k}{k^{(p-1)/p}}}$ & $\boxed{\dfrac{1}{k^{(p-1)/p}}}$  \\[-1.5ex]
(in Expectation) & & & & \\[1.5ex]
\textsc{Average Iterate} & $\dfrac{\log^2\! k}{k^{(p-1)/p}}$ & $\dfrac{1}{k^{(p-1)/p}}$ & $\dfrac{{\boxed{\log^{1/p} k}}}{k^{(p-1)/p}}$ & $\dfrac{1}{k^{(p-1)/p}}$ \\[-1.1ex]
(in Expectation) & & & & \\[1.5ex]
\textsc{Average Iterate} & $\dfrac{\log^2\! k \log(1/\delta)}{k^{(p-1)/p}}$ & $\dfrac{\log(1/\delta)}{k^{(p-1)/p}}$ & $\dfrac{{\boxed{\log^{1/p} k} } \log(1/\delta)}{k^{(p-1)/p}}$ & $\dfrac{\log(1/\delta)}{k^{(p-1)/p}}$ \\[-1.1ex]
(in high probability) & & & & \\[1.5ex]
\hline
\phantom{\bigg\vert}
\textsc{Clipping range} & $\left[2 L,+\infty\right[$ & $\left[2 L,+\infty\right[$  & $\left]L,+\infty\right[$ & $\left]L,+\infty\right[$  \\
\hline
\end{tabular}}
\vspace{2ex}
\caption{Known rates of convergence in Expectation and high probability, reported up to polyloglog factors in $\log k$, for clipped SsGM on nonsmooth problems. SOTA refers to the State-Of-The-Art. FH stands for Finite Horizon, and denotes the case where $k$ is known \emph{apriori}, while AT stands for Any Time, and refers to the case where $k$ is unknown. The SOTA rates for $p=2$ refer to the unclipped version of SsGM. The advances of this work for the rates are highlighted with boxes.
}
\label{tab:results}
\end{table}
%
\section{Problem setting and algorithm}
\label{sec:2}
In this work, we address the following optimization problem
\begin{equation}
\label{eq:mainprob}
\minimize{\xx \in \XX}{f(\xx)},
\end{equation}
under the following standing hypothesis:
\begin{enumerate}[label=H\arabic*{\rm}]
\item\label{H1} $\HH$ is a real separable\footnote{Separability is required to avoid measurability issues with vector-valued random variables.} Hilbert space and $f\colon \HH \to \bbR$ is a convex and Lipschitz continuous function with Lipschitz constant $L>0$.
\item\label{H2} The constraint set $\XX\subset \HH$ is nonempty closed and convex and 
the orthogonal projection onto $\XX$, $P_{\XX}$, can be computed explicitly.
\item\label{H2b} $\argmin_{\xx \in \XX} f(\xx) \neq \varnothing$.
\item\label{H3} $\xi$ is a random variable and, for all $\xx \in \XX$, $\hat{\ug}(\xx,\xi) \in \HH$,
$\bbE[\norm{\hat{\ug}(\xx,\xi)}]<+\infty$  and   $\bbE[\hat{\ug}(\xx,\xi)] \in \partial f(\xx)$.
\item\label{H4} For every $\xx \in \XX$, $\bbE[\norm{\hat{\ug}(\xx,\xi) - \bbE[\hat{\ug}(\xx,\xi)]}^p] \leq \sigma^p$, for some $\sigma\geq0$ and $p \in \left]1,2\right]$.
\end{enumerate}
\vspace{1ex}

Hypotheses \ref{H3} ensures that an unbias stochastic subgradient of $f$ is available, while \ref{H4}
states that the $p$-th moments of the noise are uniformly bounded.
To handle such type of situation, we consider the \emph{clipped} version of SsGM in \Cref{algorithm:clippedSsGM}, where for every $\uu \in \HH$ and $\lambda > 0$ the clipped operator is defined as
\begin{equation*}
\textsc{CLIP}(\uu, \lambda) = 
\max\big\{\norm{\uu}/\lambda , 1\big\}^{-1}\hspace{-1ex} \cdot \uu.\\
\end{equation*}
We note that in Algorithm~\ref{algorithm:clippedSsGM} 
the stepsizes $\gamma_k$'s
and clipping levels $\lambda_k$'s are left unspecified.
In the convergence analysis we will consider several explicit 
 rules for determine such parameters, depending 
 on the analyzed time horizon (number of total iterations), which can be infinite or finite,
 and whether one is interested to the averaged iterate $\bar{x}_k$
 or to the last iterate  $x_k$. In particular the following rule will cover the anytime setting 
\vspace{1ex}
 \begin{enumerate}[leftmargin=7ex, label=AT\arabic*{\rm}]
\item\label{AT}  $\displaystyle\gamma_k = \frac{\gamma}{(k (1 + \log k))^{1/p}}$,
\ $\lambda_k = \max\{\Le, \lambda (k (1 + \log k))^{1/p}\},\quad \forall\ k \in \N$,
\end{enumerate}
\vspace{1ex}
where $\Le = (1+\varepsilon)L$, with $\varepsilon>0$ arbitrary.
As for the finite-time horizon, if $k \in \N$ is the fixed horizon, we will use two alternatives.
The first one assumes a constant value along the entire time window $\{1,\dots, k\}$.
\vspace{1ex}
\begin{enumerate}[leftmargin=7ex, label=FH\arabic*{\rm}]
\item\label{FH}  $\displaystyle(\gamma_i)_{1 \leq i \leq k} \equiv \frac{\gamma}{k^{1/p}}$,
\ $(\lambda_i)_{1 \leq i\leq k} \equiv \max\{\Le,\lambda k^{1/p}\}$.
\end{enumerate}
\vspace{1ex}
The second one, based on an idea from \cite{Jain2021}, relies on dividing the execution time into epochs. Specifically,
one set $n = \lceil \log_2 k \rceil$ and, for each $j = \{0,1,\dots,n\}$, define the \emph{epoch}  $E_j := \{k_j+1,\dots,k_{j+1}\}\subset \{1,\dots, k\}$,
where
\begin{itemize}
\item $k_{n+1}=k$
\item $k_{j} = k - \lceil k/2^j \rceil$ for $j=0,1,\dots,n$.
\end{itemize}
So, the time window $\{1,\dots k\}$ is partitioned into $n+1$ epochs.
Then  the parameter values are assumed to be constant within each epoch, that is,
for every $i \in \{0,1,\dots,n\}$,
\vspace{0.2ex}
\begin{enumerate}[leftmargin=7ex, label=FH\arabic*{\rm}]
\setcounter{enumi}{1}
\item\label{FH2}  $\displaystyle(\gamma_i)_{i \in E_j} \equiv \dfrac{\gamma}{2^j k^{1/p}}$,
\ $(\lambda_i)_{i \in E_j} \equiv 2^j \max\{\Le, \lambda k^{1/p}\}$.
\end{enumerate}
\vspace{1ex}
\begin{algorithm}[t]
\caption{Clipped Stochastic subGradient Method (C-SsGM)}
\label{algorithm:clippedSsGM}
Given the stepsizes $(\gamma_k)_{k \in \N} \in \R_{++}^\N$, the clipping levels $(\lambda_k)_{k \in \N} \in \R_{++}^\N$, the batch size $m \in \N$, $m\geq 1$, and an initial point $\xx_1 \in \XX$, then
$x_1\equiv \xx_1$ and
\begin{equation}
\label{eq:algo}
\begin{array}{l}
\text{for}\;k=1,\ldots\\[1ex]
\left\lfloor
\begin{array}{l}
\text{draw } \bm{\xi}^k = (\xi_j^k)_{1 \leq j \leq m}\ \ m \text{ independent copies of } \xi\\[1ex]
\bar{\ug}_k = \displaystyle\frac 1 m \sum_{j=1}^m \hat{\ug}(x_k, \xi^k_j),\\[3ex]
\tilde{\ug}_k = 
\textsc{CLIP}(\bar{\ug}_k, \lambda_k)\\[1ex]
x_{k+1} = P_{\XX}(x_k-\gamma_k \tilde{\ug}_k).
\end{array}
\right.
\end{array}
\end{equation}
From the sequence $(x_k)_{k \in \N}$ one defines also
\begin{equation}
(\forall\, k \in \N)\qquad\bar{x}_k = \frac{1}{k} \sum_{i=1}^k x_i,\quad \ug_k = \bbE[\hat{\ug}(x_k, \xi^k_1)\,\vert\, \bm{\xi}^1, \dots, \bm{\xi}^{k-1}].
\end{equation}
Note that $(\bm{\xi}^k)_{k \in \N}$ are assumed independent too.\\
\end{algorithm}
In all settings, $\varepsilon$ shall be thought as an arbitrary small constant chosen by the user, not as a free parameter to be tuned.
\paragraph{\upshape Notation} 
We set $\bbN = \{1, 2, \dots\}$, and for every $k \in \N$, $[k] = \{1,\dots k\}$. Moreover, $\bbR_+ = [0, \infty [$ and $\bbR_{++}=\left]0,+\infty\right[$. A sequence $(\alpha_k)_{k \in \N}$ in $\R$
is \emph{decreasing} (risp.~\emph{increasing}) if $\alpha_{k+1}\leq \alpha_k$ (risp.~$\alpha_{k+1}\geq \alpha_k$) for every $k \in \N$. Given two numerical sequences $(a_k)_{k\ \in \N}$
and $(b_k)_{k\ \in \N}$ we write $a_k \lesssim b_k$ to mean $a_k \leq c b_k$, for every $k \in \N$, for some unspecified constant $c \in \R_+$.

For a real Hilbert space $\HH$ we denote by $\norm{\cdot}$ and $\scalarp{\cdot, \cdot}$ its norm and scalar product, respectively. For a convex function $f\colon \HH \mapsto \bbR$ we denote by $\partial f(\xx)$ its \emph{subdifferential} of $f$ at $\xx \in \HH$, i.e. the set
\begin{equation*}
\partial f(\xx) = \{\uu \in \HH \,\vert\, \forall \yy \in \HH, f(\yy) \geq f(\xx) + \scalarp{\uu, \yy-\xx}\}.
\end{equation*}
An element of $\partial f(\xx)$ is called \emph{subgradient} of $f$ at $\xx$.
Notice that \ref{H2} implies that all subgradients have the norm bounded by $L$. 

We will consider random variables with underlying probability space $(\Omega, \mathfrak{A}, \bbP)$ taking values in $\HH$. We use the default font for random variables and sans serif font for their realizations. The expected value operator is denoted by $\bbE$. A copy of a random variable is random variable having the same distribution of the given one. 
Finally, for a fixed integer $m \geq 1$,
we define the following recurring constants
\begin{equation}
\label{eq:constants}
\begin{aligned}
&a_p = \bigg(1 + \frac{1}{\varepsilon}\bigg)^{p-1} \sigma_m^p, \quad b_p = 2^{p-1} ( \sigma_m^p + L^p), \quad \Le = (1+\varepsilon)L,
\end{aligned}    
\end{equation}
where, $\sigma_m\geq 0$ is such that
\begin{align*}
\forall\, \xx \in \HH\colon \bbE[\|\bar{\ug} - \bbE[\bar{\ug}]\|^p] \leq \sigma_m^p,\qquad\text{and}\qquad \bar{\ug} = \frac{1}{m} \sum_{i=1}^m \hat{\ug}(\xx,\xi_i).
\end{align*}
with $(\xi_i)_{i\leq i\leq m}$ i.i.d.~copies of $\xi$. Note that $\sigma^p_m$ can always be taken as
$\sigma^p_m = C \sigma^p/m^{p-1}$ with $C \in [1,2]$
(actually when $p=2$, $C = 1$) \cite{Pinelis2015,Cox1982}.
In the rest of the paper, for the sake of brevity, we set, for every $k \in \N$,
\begin{equation*}
\mathcal{F}_k = \sigma(\bm{\xi}^1,\dots, \bm{\xi}^{k-1})\quad\text{and}\quad\bbE_{k}[\,\cdot\,] = \bbE[\,\cdot\,\vert\,\mathcal{F}_{k}]
\end{equation*}
Note that, according to Algorithm~\ref{algorithm:clippedSsGM}, the random vectors $x_k$ and $\tilde{\ug}$ are depending on the $\bm{\xi}^k$'s, more specifically we have
\begin{equation*}
x_k = x_k (\bm{\xi}^1,\dots, \bm{\xi}^{k-1})
\quad\text{and}\quad \tilde{\ug}_k = \tilde{\ug}_k (\bm{\xi}^1,\dots, \bm{\xi}^{k}).
\end{equation*}
Thus, when $\mathcal{F}_{k}$ is given, the random vectors $x_1, \dots, x_k$ (up to the current iterate) are frozen and 
\begin{equation*}
\tilde{\ug}_k=\textsc{CLIP}\bigg(\frac 1 m \sum_{j=1}^m \hat{\ug}(x_k, \xi^k_j), \lambda_k\bigg)
\end{equation*}
depends only on $\bm{\xi}^k$ (which is independent on $x_1,\dots, x_k$). 

\section{Analysis of the average iterate}
\label{sec:AV}
The convergence analysis of the average iterates of \Cref{algorithm:clippedSsGM} relies on two aspects: a bound on certain statistics of the clipped subgradient estimator throughout the iterations, and a suitable control of the distance between the iterates and the optimum. The first aspect is addressed in Lemma~\ref{lm:useful_bounds}, while the second hinges on unrolling a recursive relation---a result that is given in Lemma~\ref{lm:recurrences}. The proofs are deferred to Appendix~\ref{app:A}.

%
\begin{restatable}{lemma}{statistical}
\label{lm:useful_bounds}
The iterates generated by \Cref{algorithm:clippedSsGM} satisfy the following conditions
\begin{enumerate}[label={\rm(\roman*)}]
\item\label{eq:scalarp} $\forall\, \xx \in {\XX}\colon \bbE \scalarp{\tilde{\ug}_k-\ug_k,\xx-x_k} \leq a_p \cdot \lambda_k^{1-p} \cdot \bbE \norm{x_k-\xx}$,
\item\label{eq:norm_squared} $\bbE \norm{\tilde{\ug}_k}^2 \leq b_p \cdot \lambda_k^{2-p}$,
\end{enumerate}
for every $k \in \N$.
\end{restatable}
%

\begin{restatable}{lemma}{recurrences}
\label{lm:recurrences}
Let $(\alpha_k)_{k \in \bbN}, (\beta)_{k \in \bbN}, (b_k)_{k \in \bbN}$ and $(c_k)_{k \in \bbN}$ be sequences in $\bbR_+$ such that
\begin{align}
\label{eq:recurrences}
\forall k \in \bbN\colon \beta_k + \alpha_{k+1} \leq \alpha_k + b_k \sqrt{\alpha_k} + c_k.
\end{align}
Then the following hold for every $k \in \N$
\begin{enumerate}[label={\rm(\roman*)}]
\item\label{lm:recurrences_ii} $\displaystyle \max_{1\leq i\leq k+1} \sqrt{\alpha_{i}} \leq \sum_{i=1}^k b_i + \sqrt{\alpha_1 + \sum_{i=1}^k c_i}$.
\item\label{lm:recurrences_i} $\displaystyle \sum_{i=1}^k \beta_i + \alpha_{k+1} \leq \frac{10}{9} \bigg(\alpha_1 + \sum_{i=1}^k c_i\bigg) + \bigg(\sum_{i=1}^k b_i\bigg)^{\!2}$.
\end{enumerate}
\end{restatable}

\begin{lemma}[A standard inequality]\label{lem:standard}
The iterates of \Cref{algorithm:clippedSsGM}, satisfy the following inequality
\begin{align}
\label{eq:standardineq}
\gamma_k (f(x_k)-f(\xx)) &\leq \frac{\norm{x_k-\xx}^2-\norm{x_{k+1}-\xx}^2}{2} + \gamma_k \scalarp{\tilde{\ug}_k-\ug_k,\xx-x_k} + \frac{\gamma_k^2}{2} \norm{\tilde{\ug}_k}^2,
\end{align}
for every $k \in \N$ and $\xx \in \XX$. Moreover, every $k \in \N$, $\norm{x_k}^2$ and $f(x_k)$ are summable random variables.
\end{lemma}
\begin{proof}
Let $\xx \in \XX$. It follows from the definition of $x_{k+1}$ and the nonexpansivity of $P_{\hspace{0.1ex}\XX}$ that
\begin{align}
\label{eq:20250322a}
\nonumber\norm{x_{k+1}-\xx}^2 &= \norm{P_{\hspace{0.15ex}\XX}(x_k - \gamma_k \tilde{\ug}_k)- P_{\hspace{0.15ex}\XX}(\xx)}^2 
\leq \norm{x_k -\xx - \gamma_k \tilde{\ug}_k}^2\\[1ex]
\nonumber&=\norm{x_k-\xx}^2 +2\gamma_k\scalarp{\xx-x_k,\tilde{\ug}_k} + \gamma_k \norm{\tilde{\ug}_k}^2\\[1ex]
&=\norm{x_k-\xx}^2 +2\gamma_k\scalarp{\xx-x_k,\tilde{\ug}_k- \ug_k} +2\gamma_k\scalarp{\xx-x_k, \ug_k}  + \gamma_k \norm{\tilde{\ug}_k}^2.
\end{align}
Now, since $u_k \in \partial f(x_k)$, $f(\xx)\geq f(x_k) + \scalarp{\xx-x_k, \ug_k}$ and hence
\begin{equation*}
\norm{x_{k+1}-\xx}^2 \leq 
\norm{x_k-x_*}^2 +2\gamma_k\scalarp{\xx-x_k,\tilde{\ug}_k- \ug_k} + 2\gamma_k (f(\xx)-f(x_k))  + \gamma_k \norm{\tilde{\ug}_k}^2.
\end{equation*}
Rearranging the terms the inequality follows.
Concerning the second part of the statement, we prove by induction that for every 
$k \in \N$, $\norm{x_k}^2$ is summable. Indeed the
statement is clearly true for $k=1$. Suppose that it is true for $k \in \N$. Then 
it follows from \eqref{eq:20250322a} that
\begin{equation*}
\norm{x_{k+1}-\xx}^2 \leq \norm{x_k-\xx}^2 +2\gamma_k\norm{\xx-x_k}\norm{\tilde{\ug}_k- \ug_k} +2\gamma_k\norm{\xx-x_k}\norm{\ug_k}  + \gamma_k \norm{\tilde{\ug}_k}^2.
\end{equation*}
Since $\norm{\tilde{\ug}_k- \ug_k}$ and $\norm{\ug_k}$ are bounded
and $\norm{\xx-x_k}$ is summable (being square summable),
it follows that $\norm{x_{k+1}-\xx}^2$ is summable. Finally the summability of $f(x_k)$
follows from \eqref{eq:standardineq} and the fact that $f$ is bounded from below on $\XX$ (being $\argmin_{\XX} f\neq \varnothing$).
\end{proof}



\subsection{Convergence in expectation}

We are now ready to prove the first main result of this section.
\begin{theorem}
\label{thm:avg_unbounded}
Referring to Algorithm~\ref{algorithm:clippedSsGM}, suppose that
 $(\gamma_k)_{k \in \bbN}$ is decreasing and that $(\lambda_k)_{k \in \bbN}$ is such that $\lambda_k \geq \Le$ for every $k \in \bbN$. Let  $\xx_* \in \argmin_\XX\hspace{-0.1ex} f$, $f_* = \min_{\XX} f$ and set, for every $k \in \N$,
 \begin{equation*}
A_k = a_p \sum_{i=1}^k \frac{\gamma_i}{\lambda_i^{p-1}},
\qquad B_k = b_p \sum_{i=1}^k \gamma_i^2 \lambda_i^{2-p},
\end{equation*}
where $a_p$ and $b_p$ are defined in equation \eqref{eq:constants}.
Then, for every $k \in \N$, the following hold.
\begin{enumerate}[label={\rm(\roman*)}]
 \item\label{eq:avg_unbounded} $\displaystyle \bbE [f(\bar{x}_k)]  - f_* 
 \leq \frac{1}{k} \sum_{i=1}^k \bbE[f(x_i)-f_*]
 \leq \frac{1}{k \gamma_k} \bigg[ \frac 5 9 \big( \norm{x_{1}-\xx_*}^2 + B_k\big) + 2 A_k^2 \bigg]$.
\item\label{eq:distance_bound} $\displaystyle \max_{i \in [k]}\bbE \norm{x_{i+1}-\xx_*} \leq  \norm{x_{1}-\xx_*} + 2 A_k
+ \sqrt{B_k}$\,.
\end{enumerate}
%
\end{theorem}
\begin{proof}
\ref{eq:avg_unbounded}:
It follows from the standard inequality in Lemma~\ref{lem:standard} with $\xx=\xx_*\in \argmin_{\XX}\hspace{-0.1ex} f$,
that
\begin{align}
\label{eq:standardineq2}
\gamma_k (f(x_k)-f_*) &\leq \frac{\norm{x_k-\xx_*}^2-\norm{x_{k+1}-\xx_*}^2}{2} + \gamma_k \scalarp{\xx_*-x_k, \tilde{\ug}_k-\ug_k} + \frac{\gamma_k^2}{2} \norm{\tilde{\ug}_k}^2. 
\end{align}
Taking the expectation and recalling \Cref{lm:useful_bounds}, with $x = x_*$, and definitions \eqref{eq:constants}, we have
\begin{align*}
\gamma_k \bbE[f(x_k)-f_*] &\leq \frac{\bbE \norm{x_k-\xx_*}^2 - \bbE \norm{x_{k+1}-\xx_*}^2}{2} + \gamma_k \bbE\scalarp{\tilde{\ug}_k-\ug_k, \xx_*-x_k}\! +\! \frac{\gamma_k^2}{2}  \bbE \norm{\tilde{\ug}_k}^2 \\[1ex]
&\leq \frac{\bbE \norm{x_k-\xx_*}^2 - \bbE \norm{x_{k+1}-\xx_*}^2}{2} + a_p \gamma_k \frac{\bbE \norm{x_k-\xx_*}}{\lambda_k^{p-1}} + \frac{\gamma_k^2 b_p \lambda_k^{2-p}}{2} \\
&\leq \frac{\bbE \norm{x_k-\xx_*}^2 - \bbE \norm{x_{k+1}-\xx_*}^2}{2} + a_p \gamma_k \frac{\sqrt{\bbE \norm{x_k-\xx_*}^2}}{\lambda_k^{p-1}} + \frac{b_p \gamma_k^2 \lambda_k^{2-p}}{2} ,
\end{align*}
where in the last inequality we used that $\bbE Z \leq \sqrt{\bbE Z^2}$ for any positive random variable $Z$. Thus, by \Cref{lm:recurrences} with
\begin{align*}
\alpha_k = \frac 1 2 \bbE \norm{x_k-\xx_*}^2,\ \ \beta_k = \gamma_k \bbE[f(x_k)-f_*],\ \  b_k = \sqrt{2} a_p \gamma_k/\lambda_k^{p-1},\ \  c_k = \frac 1 2 b_p \gamma_k^2 \lambda_k^{2-p},
\end{align*}
we have
\begin{align}
\label{eq:thm_avg_unbounded_p1}
\nonumber\sum_{i=1}^k \gamma_i \bbE[f(x_i) - f_*] &+ \frac{1}{2} \bbE\norm{x_{k+1}-\xx_*}^2\\ 
&\leq \frac 5 9\bigg(\bbE \norm{x_1-\xx_*}^2 + b_p \sum_{i=1}^k \gamma_i^2 \lambda_i^{2-p} \bigg)+ 2 a_p^2 \Bigg(\sum_{i=1}^k \frac{\gamma_i}{\lambda_i^{p-1}} \Bigg)^{\!2}.
\end{align}
Moreover, since $f$ is convex and $(\gamma_k)_{k \in \bbN}$ is decreasing, we have
\begin{align*}
k \gamma_k \bbE[f(\bar{x}_k)-f_*] &\leq k \gamma_k \cdot \frac{1}{k} \sum_{i=1}^k \bbE[f(x_i)-f_*] \leq \sum_{i=1}^k \gamma_i \bbE[f(x_i)-f_*].
\end{align*}
and hence \ref{eq:avg_unbounded} follows.

\ref{eq:distance_bound}:
Using the previous notation, it follows from Lemma~\ref{lm:recurrences}\ref{lm:recurrences_ii}
by noting that, $\bbE \norm{x_{i+1}-\xx_*} \leq \sqrt{\bbE \norm{x_{i+1}-\xx_*}^2}$
and $\sqrt{\bbE \norm{x_{1}-\xx_*}^2} = \norm{x_1 - \xx_*}$ (being $x_1$ not random).
\end{proof}
\begin{remark}
\label{rmk:avg_unbounded}
The following comments are in order.
\begin{enumerate}[label={\rm(\roman*)}]
\item \Cref{thm:avg_unbounded}\ref{eq:avg_unbounded} provides a general convergence result for the expected values of the function $f$. Explicit rates of convergence can be obtained as long as
the sequences
\begin{equation*}
\frac{1}{k \gamma_k }, \quad \frac{1}{k\gamma_k} \Bigg( \sum_{i=1}^k \frac{\gamma_i}{\lambda_i^{p-1}} \Bigg)^2\!, 
\ \text{and}\ \ \frac{1}{k\gamma_k} \sum_{i=1}^k \gamma_i^2 \lambda_i^{2-p},
\end{equation*}
can be properly bounded.

\item \Cref{thm:avg_unbounded}\ref{eq:distance_bound} provides an \emph{a posteriori} bound on the distance between the iterates and any optimum, that will be critical for proving the convergence of the last iterate. 
\end{enumerate}
\end{remark}

We now specify the stepsizes and the clipping levels for both the anytime and finite horizon settings respectively. Plugging the rules \ref{AT} and \ref{FH} in Theorem~\ref{thm:avg_unbounded}
we obtain the following result (whose proof is in the Appendix~\ref{app:A}).
\begin{restatable}{corollary}{avgunbounded}
\label{crl:avg_unbounded}
Under the assumptions of Theorem~\ref{thm:avg_unbounded},
let $\bar{c} \coloneqq 2 (a_p/\lambda^{p-1})^{2} + b_p \max\{\Le, \lambda\}^{2-p}$. Then the following rates hold.
\begin{enumerate}[label={\rm(\roman*)}]
\item\label{crl:avg_unbounded_at}\textbf{Anytime.} Suppose that $\gamma_k$  and $\lambda_k $ are as in {\rm\ref{AT}}.
Then, for every $k\in \N$,
\begin{equation}
\label{eq:avg_unbounded_at}
\bbE[f(\bar{x}_k)]  - f_*
\leq \frac{(1+\log k)^{1/p} \cdot (1+\log (1 + \log k))^2 }{k^{(p-1)/p}} \Bigg[\frac{\norm{x_{1}-\xx_*}^2}{\gamma} + \gamma\bar{c}  \Bigg].
\end{equation}
%
\item\label{crl:avg_unbounded_fh}\textbf{Finite-horizon.} 
Suppose that $\gamma_k$  and $\lambda_k $ are as in {\rm\ref{FH}}.
Then, for every $k\in \N$,
\begin{equation}
\label{eq:avg_unbounded_fh}
\bbE[f(\bar{x}_k)]  - f_*
\leq \frac{1}{k^{(p-1)/p}} \Bigg[\frac{\norm{x_{1}-\xx_*}^2}{\gamma} + \gamma\bar{c}  \Bigg].
\end{equation}
%
\end{enumerate}
\end{restatable}
\begin{remark}
The following comments are in order.
%
\begin{enumerate}[label={\rm(\roman*)}]
\item The rate in the finite horizon setting is of order  $1/k^{(p-1)/p}$ which is optimal as it matches (up to constants) the lower bound in \cite{Vural2022}. 
\item The bound \eqref{eq:avg_unbounded_at} improves upon the result in \cite{Liu2023b}, that instead features a convergence rate of order  $(\log^2 k)/k^{(p-1)/p}$. Furthermore, our parameter setting allows any clipping level in the range $\left]L,\infty\right[$, whereas \cite{Liu2023b} contraints the clipping level in the interval $\left[2L,\infty\right[$, which is limiting in practice---especially when $L$ is large as is often the case in high-dimensions. We demonstrate the advantages of our parameter setting in the numerical experiments section.
\item The rate in the anytime setting is (ignoring polylogarithmic factors in $\log k$) of order  $(\log^{1/p} k)/k^{(p-1)/p}$, which is worse by a $\log^{1/p} k$ factor than its finite-horizon counterpart. It remains unclear if this factor is truly necessary, or if a more careful analysis and/or a different parameter settings could eliminate it. Indeed, our experimental results suggest that an anytime parameter setting exhibits superior performance in practice.
\end{enumerate}
\end{remark}

\subsection{Convergence in high probability}

In this section we provide an equivalent of \Cref{thm:avg_unbounded} in the case of convergence in high probability. We start by considering the following (standard) implication of the Freedman inequality \cite{Freedman1975,Fan2025}.
\begin{proposition}[Freedman's inequality]
\label{prop:FreedmanBound}
Let $(\zeta_i)_{1\leq i\leq k}$ be a martingale difference sequence\footnote{This means that $\bbE[\zeta_i]<+\infty$ and $\bbE[\zeta_i\,\vert\, \zeta_1,\dots, \zeta_{i-1}]=0$, for every $1\leq i\leq k$, with $\zeta_0=0$.} such that for all $i \in [k]$
\begin{enumerate}[label={\rm(\roman*)}]
\item $\zeta_i \leq c$ a.s.,\ $(c>0)$
\item $\sigma_i^2 \coloneqq \bbE[\zeta_i^2 |\zeta_1,\dots,\zeta_{i-1}] < \infty$ a.s.
\end{enumerate}
Let $\delta \in \left]0, 1\right[$, set $V_k = \sum_{i=1}^k \sigma_i^2$, and let $F\geq 0$ s.t. $V_k \leq F$ almost surely. Then, with probability at least  $1-\delta/2$, we have
\begin{align*}
\forall\, j \in [k]\colon \sum_{i=1}^j \zeta_i \leq \frac{2}{3} c \log\left(\frac{2}{\delta}\right) + \sqrt{ 2 F\log\left(\frac{2}{\delta}\right)}.
\end{align*}
\end{proposition}
We can now prove the following convergence rate in high probability.

%
\begin{theorem}
\label{thm:avg_unbounded_hp}
Referring to Algorithm~\ref{algorithm:clippedSsGM}, suppose that
 $(\gamma_k)_{k \in \bbN}$ is decreasing and that $(\lambda_k)_{k \in \bbN}$ is such that $\lambda_k \geq \Le$ for every $k \in \bbN$. Let  $\xx_* \in \argmin_{\XX}\hspace{-0.1ex} f$ and $f_* = \min_{\XX} f$. Then for $\delta \in\left]0,1\right[$ and for every $k \in \bbN$, with probability at least $1-\delta$ it holds 
\begin{align}
\label{eq:avg_unbounded_hp}
f(\bar{x}_k) -f_* \leq \frac{1}{2 k\gamma_k}  & \Bigg( \norm{x_{1}-\xx_*} + A_k + B_k \nonumber\\
&\qquad+\frac{8}{3}\big(C_k + D_k \big) \cdot \log \Big(\frac{2}{\delta}\Big) + \Big(\sqrt{B_k} + \sqrt{E_k} \Big) \cdot \sqrt{\log \Big(\frac{2}{\delta}\Big)} \Bigg)^2,
\end{align}
where $A_k$ and $B_k$ are defined in \Cref{thm:avg_unbounded}, 
\begin{align*}
C_k = \max_{i \in [k]}\gamma_i \lambda_i, \quad D_k = \max_{i \in [k]}\gamma_i^2 \lambda_i^2, \quad E_k = b_p \sum_{i=1}^k \gamma_i^4 \lambda_i^{4-p},
\end{align*}
and $a_p$ and $b_p,$ are defined in equation \eqref{eq:constants}.
\end{theorem}
\begin{proof}
Let $k \in \N$.
Set, for the sake of brevity, for every $i \in [k]$, $\Delta(x_i) = f(x_i)-f_*$. It follows from the standard inequality \eqref{eq:standardineq} 
and item \ref{eq:norm_squared} in Lemma~\ref{lm:useful_bounds}, that
%
\begin{align}
\label{eq:step1}
2 \gamma_i \Delta(x_i) + \norm{x_{i+1}-\xx_*}^2 &\leq \norm{x_{i}-\xx_*}^2 + 2 \gamma_i \scalarp{\xx_*-x_i, \tilde{\ug}_i-\ug_i} + \gamma_i^2 \norm{\tilde{\ug}_i}^2 \nonumber \\[1ex]
&= \norm{x_{i}-\xx_*}^2 + 2 \gamma_i \big(\scalarp{\xx_*-x_i, \tilde{\ug}_i-\bbE_i \tilde{\ug}_i} + \scalarp{\xx_*-x_i, \bbE_i \tilde{\ug}_i-\ug_i}\big)  \nonumber \\[1ex]
&\qquad+ \gamma_i^2 (\norm{\tilde{\ug}_i}^2-\bbE_i \norm{\tilde{\ug}_i}^2 + \bbE_i \norm{\tilde{\ug}_i}^2) \nonumber\\[1ex]
&\leq \norm{x_{i}-\xx_*}^2 + 2 \gamma_i \big(\scalarp{\xx_*-x_i, \tilde{\ug}_i-\bbE_i \tilde{\ug}_i} + \scalarp{\xx_*-x_i, \bbE_i \tilde{\ug}_i-\ug_i}\big)\nonumber \\[1ex]
&\qquad + \gamma_i^2 \big[(\norm{\tilde{\ug}_i}^2-\bbE_i \norm{\tilde{\ug}_i}^2) + b_p \cdot \lambda_i^{2-p}\big].
\end{align}
Now, we define $\cD_i = \max\{\norm{x_{1}-\xx_*},\dots,\norm{x_{i}-\xx_*}, 1\}$, and notice that $\cD_i \geq 1$ for each $i \in [k]$. Thus multiplying both sides of \eqref{eq:step1} by $1/\cD_i$ yields
\begin{align}
\label{eq:step_2}
\nonumber2 \gamma_i \frac{\Delta(x_i)}{\cD_i} &+ \frac{\norm{x_{i+1}-\xx_*}^2}{\cD_i}\\[1ex]
&\leq \frac{\norm{x_i-\xx_*}^2}{\cD_i} + 2 \gamma_i \bigg(\Big\langle \frac{\xx_*-x_i}{\cD_i}, \tilde{\ug}_i-\bbE_i \tilde{\ug}_i\Big\rangle 
+ \Big\langle\frac{\xx_*-x_i}{\cD_i},\bbE_i \tilde{\ug}_i-\ug_i\Big\rangle\bigg)\nonumber \\[1ex]
&\ \qquad\qquad\qquad + \frac{\gamma_i^2}{\cD_i} [(\norm{\tilde{\ug}_i}^2-\bbE_i \norm{\tilde{\ug}_i}^2) + b_p \cdot \lambda_i^{2-p}] \nonumber \\[1ex]
&\leq \frac{\norm{x_i-\xx_*}^2}{\cD_i} + 2 \gamma_i \bigg(\Big\langle \frac{\xx_*-x_i}{\cD_i}, \tilde{\ug}_i-\bbE_i \tilde{\ug}_i\Big\rangle 
+ \Big\langle\frac{\xx_*-x_i}{\cD_i},\bbE_i \tilde{\ug}_i-\ug_i\Big\rangle\bigg)\nonumber \\[1ex]
&\ \qquad\qquad\qquad + \gamma_i^2 [(\norm{\tilde{\ug}_i}^2-\bbE_i \norm{\tilde{\ug}_i}^2) + b_p \cdot \lambda_i^{2-p}].
\end{align}
Next, let $j\in [k]$. Since $(\cD_i)_{i\in[j]}$ is increasing, summing all the inequalities in \eqref{eq:step_2} from $1$ to $j$ yields\footnote{Recall that for every sequence
 $(c_i)_{i \in [j+1]}\in \R_+^{j+1}$ and decreasing sequence $(\alpha_i)_{i \in [j]}\in \R_+^j$, we have
\begin{equation*}
\sum_{i=1}^j (\alpha_i c_{i}-\alpha_i c_{i+1}) = a_1 c_1 +\sum_{i=2}^{j-1} \alpha_i c_{i} - 
\sum_{i=1}^{j-1} \alpha_i c_{i+1} - \alpha_j c_{j+1}= 
a_1 c_1- \alpha_j c_{j+1} + \sum_{i=2}^j (\underbrace{\alpha_i - \alpha_{i-1}}_{\leq 0})c_i
\end{equation*}}
\begin{align*}
2 \sum_{i=1}^j \gamma_i &\frac{\Delta(x_i)}{\cD_i} + \frac{\norm{x_{j+1}-\xx_*}^2}{\cD_j}\\
& \leq \frac{\norm{x_1-\xx_*}^2}{\cD_1}+ 2 \sum_{i=1}^j \gamma_i \bigg(\Big\langle\tilde{\ug}_i-\bbE_i [\tilde{\ug}_i],\frac{x_i-\xx_*}{\cD_i}\Big\rangle + \Big\langle\bbE_i [\tilde{\ug}_i]-\ug_i,\frac{x_i-\xx_*}{\cD_i}\Big\rangle\bigg) \\[1ex]
&\ \qquad\qquad\qquad+ \sum_{i=1}^j \gamma_i^2 \big[(\norm{\tilde{\ug}_i}^2-\bbE_i \norm{\tilde{\ug}_i}^2) + b_p \cdot \lambda_i^{2-p}\big] \\[1ex]
&\leq \norm{x_1-\xx_*}+ 2 \sum_{i=1}^j \gamma_i \bigg(\Big\langle\tilde{\ug}_i-\bbE_i [\tilde{\ug}_i],\frac{x_i-\xx_*}{\cD_i}\Big\rangle + \Big\langle\bbE_i [\tilde{\ug}_i]-\ug_i,\frac{x_i-\xx_*}{\cD_i}\Big\rangle\bigg) \\[1ex]
&\ \qquad\qquad\qquad+ \sum_{i=1}^j \gamma_i^2 \big[(\norm{\tilde{\ug}_i}^2-\bbE_i \norm{\tilde{\ug}_i}^2) + b_p \cdot \lambda_i^{2-p}\big], 
\end{align*}
where in the last inequality we used the definition of $\cD_1$. Finally, since $\cD_j \geq \cD_i$, we obtain
\begin{align}
\label{eq:step3}
2 \sum_{i=1}^j \gamma_i \Delta(x_i) + \norm{x_{j+1}-\xx_*}^2 \leq \cD_j \Bigg(&\norm{x_1-\xx_*} + 2 \sum_{i=1}^j \gamma_i (\textsc{Dev}(\tilde{\ug}_i) + \textsc{Bias}(\tilde{\ug}_i)) \nonumber \\
&\ \qquad\qquad+ \sum_{i=1}^j \gamma_i^2 \big(\textsc{Dev}^2(\tilde{\ug}_i) + b_p \cdot \lambda_i^{2-p}\big)\Bigg),
\end{align}
where we defined
\begin{align*}
\textsc{Bias}(\tilde{\ug}_i) &= \Big\langle\bbE_i [\tilde{\ug}_i] - \ug_i,\frac{x_i-\xx_*}{\cD_i}\Big\rangle, \\
\textsc{Dev}(\tilde{\ug}_i) &= \Big\langle\tilde{\ug}_i-\bbE_i [\tilde{\ug}_i],\frac{x_i-\xx_*}{\cD_i}\Big\rangle, \\[1ex]
\textsc{Dev}^2(\tilde{\ug}_i) &= \norm{\tilde{\ug}_i}^2-\bbE_j \norm{\tilde{\ug}_i}^2.
\end{align*}
Now we analyze each of the terms in the right hand side of \eqref{eq:step3} separately. 

\noindent\textbf{$\textsc{Bias}(\tilde{\ug}_i)$.} For the bias term, applying \Cref{lm:stat_prop}\ref{lm:stat_prop_bias}
and noting that $\cD_i \geq \norm{x_i- \xx_*}$, one has
\begin{align}
\label{eq:step4_1}
\sum_{i=1}^j \gamma_i \textsc{Bias}(\tilde{\ug}_i) = \sum_{i=1}^j \gamma_i \scalarp{\bbE_i \tilde{\ug}_i - \ug_i,\frac{x_i-\xx_*}{\cD_i}} \leq \sum_{i=1}^j \gamma_i \norm{\bbE_i \tilde{\ug}_i - \ug_i} \leq a_p \cdot \sum_{i=1}^k \gamma_i \lambda_i^{1-p}.
\end{align}
\textbf{$\textsc{Dev}(\tilde{\ug}_i)$.} 
Since $ \gamma_i \textsc{Dev}(\tilde{\ug}_i) = \gamma_i \scalarp{\tilde{\ug}_i-\bbE_i[\tilde{\ug}_i],(x_i-\xx_*)/\cD_i}$, and $\cD_i$ depends on $x_1,\dots,x_i$, we have that $\textsc{Dev}(\tilde{\ug}_i)$ is measurable w.r.t.~$\mathcal{F}_{i+1}$.
Therefore, setting $\mathcal{G}_i =\sigma(\textsc{Dev}(\tilde{\ug}_1),\dots, \textsc{Dev}(\tilde{\ug}_i))$, we have $\mathcal{G}_i \subset \mathcal{F}_{i+1}$ and hence
\begin{equation*}
\bbE[\gamma_i \textsc{Dev}(\tilde{\ug}_i)\,\vert\, \mathcal{G}_{i-1}]
= \bbE\big[\underbrace{\bbE[\gamma_i \textsc{Dev}(\tilde{\ug}_i)\,\vert\, \mathcal{F}_{i}]}_{=0} \,\vert\, \mathcal{G}_{i-1}\big] =0.
\end{equation*}
Moreover
\begin{align*}
|\gamma_i  \textsc{Dev}(\tilde{\ug}_i)| \leq \gamma_i \norm{\tilde{\ug}_i-\bbE_i [\tilde{\ug}_i]} \leq 2 \cdot \gamma_i \lambda_i \leq 2\cdot\max_{i \in [k]} \gamma_i \lambda_i \eqqcolon c, 
\end{align*}
and, by  \Cref{lm:stat_prop}\ref{lm:stat_prop_var}, it also holds
\begin{align*}
\sigma_i^2 :&= \bbE\big[(\gamma_i \textsc{Dev}(\tilde{\ug}_i))^2\,\big\vert\, \mathcal{G}_{i-1}\big]
\leq \gamma_i^2 \bbE\big[\norm{\tilde{\ug}_i-\bbE_i [\tilde{\ug}_i]}^2\,\big\vert\, \mathcal{G}_{i-1}\big]\\
&=\gamma_i^2 \bbE\big[\bbE_i[\norm{\tilde{\ug}_i-\bbE_i [\tilde{\ug}_i]}^2]\,\big\vert\, \mathcal{G}_{i-1}\big]
 \leq b_p \cdot \gamma_i^2 \lambda_i^{2-p}. 
\end{align*}
Then, for $\delta \in (0, 2/e]$, by Freedman inequality it follows that with probability at least $1-\delta/2$ it holds
\begin{align}
\label{eq:step4_2}
\forall\,j \in [k]\colon
\sum_{i=1}^j \gamma_i \textsc{Dev}(\tilde{\ug}_i) \leq \frac{4}{3} \bigg(\max_{i \in [k]} \gamma_i \lambda_i \bigg) \cdot \log \frac{2}{\delta} + \sqrt{2 b_p \cdot \sum_{i=1}^k \gamma_i^2 \lambda_i^{2-p} \cdot \log \frac{2}{\delta}}.
\end{align}
\textbf{$\textsc{Dev}^2(\tilde{\ug}_i)$.} Since $ \gamma_i^2 \textsc{Dev}^2(\tilde{\ug}_i) = \gamma_i^2 (\norm{\tilde{\ug}_i}^2-\bbE_i [\norm{\tilde{\ug}_i}^2])$, it is clear that $\gamma_i^2 \textsc{Dev}^2(\tilde{\ug}_i)$ is measurable w.r.t.~$\mathcal{F}_{i+1}$
and hence $\mathcal{G}_i := \sigma(\textsc{Dev}(\norm{\tilde{\ug}_1}^2,\dots, \textsc{Dev}(\norm{\tilde{\ug}_i}^2) \subset \mathcal{F}_i$. Thus,
\begin{equation*}
\bbE\big[ \gamma_i^2 \textsc{Dev}^2(\tilde{\ug}_i)\,\vert\, \mathcal{G}_{i-1}\big] =
\bbE\big[ \underbrace{\bbE_i \big[\gamma_i^2 \textsc{Dev}(\norm{\tilde{\ug}_i}^2)\big]}_{=0} \,\vert\, \mathcal{G}_{i-1}\big] =0.
\end{equation*}
Moreover
\begin{align*}
|\gamma_i^2 \textsc{Dev}(\norm{\tilde{\ug}_i}^2)| 
= \gamma_i^2 \big\lvert\norm{\tilde{\ug}_i}^2-\bbE_i \norm{\tilde{\ug}_i}^2\big\rvert \leq 2 \cdot \gamma_i^2 \lambda_i^2 \leq 2\cdot\max_{i \in [k]} \gamma_i^2 \lambda_i^2 \eqqcolon c,
\end{align*}
and, by \Cref{lm:stat_prop}\ref{lm:stat_prop_moment}, it also holds
\begin{align*}
\sigma_i^2 &\coloneqq \bbE\big[ (\gamma_i^2 \textsc{Dev}(\norm{\tilde{\ug}_i}^2))^2\,\big\vert\, \mathcal{G}_{i-1}\big]
= \gamma_i^4\bbE\big[ (\norm{\tilde{\ug}_i}^2-\bbE_i \norm{\tilde{\ug}_i}^2)^2 \,\big\vert\, \mathcal{G}_{i-1}\big]\\
&= \gamma_i^4 \bbE\big[\bbE_i [(\norm{\tilde{\ug}_i}^2-\bbE_i \norm{\tilde{\ug}_i}^2)^2] \,\big\vert\, \mathcal{G}_{i-1}\big]\\
&= \gamma_i^4 \bbE\big[\bbE_i \norm{\tilde{\ug}_i}^4 \,\big\vert\, \mathcal{G}_{i-1}\big] = \gamma_i^4 \bbE\big[\bbE_i [\norm{\tilde{\ug}_i}^2 \norm{\tilde{\ug}_i}^2] \,\big\vert\, \mathcal{G}_{i-1}\big]  \leq b_p \cdot \gamma_i^4 \lambda_i^{4-p}. 
\end{align*}
Then, for $\delta \in \left]0, 2/e\right]$, by Freedman inequality it follows that that with probability at least $1-\delta/2$ it holds
\begin{align}
\label{eq:step4_3}
\forall\,j \in [k]\colon
\sum_{i=1}^j \gamma_i \textsc{Dev}(\norm{\tilde{\ug}_i}^2) \leq \frac{4}{3} \bigg(\max_{i \in [k]} \gamma_i^2 \lambda_i^2 \bigg) \cdot \log \frac{2}{\delta} + \sqrt{2 b_p \cdot \sum_{i=1}^k \gamma_i^4 \lambda_i^{4-p} \cdot \log \frac{2}{\delta}}.
\end{align}
Thus combining \eqref{eq:step4_1}, \eqref{eq:step4_2}, and \eqref{eq:step4_3} into \eqref{eq:step3}, we showed that with probability at least $1-\delta$, we have that for every $j \in [k]$
\begin{align}
\label{eq:step4}
2 \sum_{i=1}^j \gamma_i \Delta(x_i) + \norm{x_{j+1}-\xx_*}^2 &\leq \cD_k \Bigg[\norm{x_1-\xx_*} + \frac{8}{3} \Bigg(  \max_{i \in [k]} \gamma_i \lambda_i + \max_{i \in [k]} \gamma_i^2 \lambda_i^2 \Bigg) \cdot \log \frac{2}{\delta} \nonumber\\
&\quad\qquad+\Bigg( \sqrt{2 b_p \cdot \sum_{i=1}^k \gamma_i^2 \lambda_i^{2-p}} + \sqrt{2 b_p \cdot \sum_{i=1}^k \gamma_i^4 \lambda_i^{4-p}} \Bigg) \cdot \sqrt{\log \frac{2}{\delta}} \nonumber\\
&\quad\qquad+ a_p \cdot \sum_{i=1}^k \gamma_i \lambda_i^{1-p} + b_p \cdot \sum_{i=1}^k \gamma_i^2 \lambda_i^{2-p} \Bigg].    
\end{align}
Next, we define
\begin{align*}
\cC_k = \norm{x_1-\xx_*} &+ \frac{8}{3} \Bigg(  \max_{i \in [k]} \gamma_i \lambda_i + \max_{i \in [k]} \gamma_i^2 \lambda_i^2 \Bigg) \cdot \log \frac{2}{\delta} \nonumber\\
&+\Bigg( \sqrt{b_p \cdot \sum_{i=1}^k \gamma_i^2 \lambda_i^{2-p}} + \sqrt{b_p \cdot \sum_{i=1}^k \gamma_i^4 \lambda_i^{4-p}} \Bigg) \cdot \sqrt{\log \frac{2}{\delta}} \nonumber\\
&+ a_p \cdot \sum_{i=1}^k \gamma_i \lambda_i^{1-p} + b_p \cdot \sum_{i=1}^k \gamma_i^2 \lambda_i^{2-p}.
\end{align*}
Thus, assuming that the event $\cE$ in \eqref{eq:step4} happens, we have
\begin{align*}
\forall\, j\in [k]\colon\ \norm{x_{j+1}-\xx_*}^2 &\leq \cD_k \cdot \cC_k
\end{align*}
This implies that $\cD_k^2\leq \cD_k\cdot\cC_k\ \Rightarrow\ \cD_k\leq \cC_k$ and hence
%
\begin{align*}
2 \sum_{i=1}^k \gamma_i \Delta(x_i)
\leq  \cC_k^2\ \ \text{on}\ \cE.
\end{align*}
Finally, noting that $(\gamma_k)_{k \in \bbN}$ is decreasing and using the convexity of $f$, one obtains
that with probability at least $1- \delta$ we have
\begin{align*}
\Delta(\bar{x}_k) \leq \frac{1}{k} \sum_{i=1}^k \Delta(x_i) \leq \frac{\cC_k^2}{2 \gamma_k k}
\end{align*}
and the statement follows.
\end{proof}

\begin{remark}
\Cref{thm:avg_unbounded_hp} is an analog of \Cref{thm:avg_unbounded} in terms of convergence in high probability. In particular, in addition to the conditions in \Cref{rmk:avg_unbounded}, it also requires to control the following  sequences 
\begin{align*}
\frac{1}{k \gamma_k } \bigg(\max_{i \in [k]}\gamma_i \lambda_i\bigg)^2,\ \ \frac{1}{k \gamma_k } \bigg(\max_{i \in [k]}\gamma_i^2 \lambda_i^2\bigg)^2,\ \ \frac{1}{k \gamma_k } \bigg(\sum_{i=1}^k \gamma_i^2 \lambda_i^{2-p}\bigg)^2,\ \text{and}\ \frac{1}{k \gamma_k } \sum_{i=1}^k \gamma_i^4 \lambda_i^{4-p}.
\end{align*}
\end{remark}
Plugging the rules \ref{AT} and \ref{FH} in Theorem~\ref{thm:avg_unbounded} we have the following result (whose proof is in the Appendix~\ref{app:A}).
\begin{restatable}{corollary}{crlavgunboundedhp}
\label{crl:avg_unbounded_hp}
Under the assumptions of Theorem~\ref{thm:avg_unbounded}, the following rates hold.
\begin{enumerate}[label={\rm(\roman*)}]
\item\label{crl:avg_unbounded_hp_at}\textbf{Anytime.} Suppose that $\gamma_k$  and $\lambda_k $ are as in {\rm\ref{AT}} and that $\gamma = 1/\log(2/\delta)$.
Then for $\delta \in\left]0,2/e\right]$ and for every $k \in \bbN$, with probability at least $1-\delta$ it holds that
\begin{equation}
\label{eq:avg_unbounded_hp_at}
f(\bar{x}_k)  - f_*
\lesssim \frac{(1+\log k)^{1/p} }{k^{(p-1)/p}} \bigg((1+\log (1+\log k))^2+\log\bigg(\frac{2}{\delta}\bigg)\bigg).
\end{equation}
\item\label{crl:avg_unbounded_hp_fh}\textbf{Finite-horizon.} Suppose that $\gamma_k$  and $\lambda_k $ are as in {\rm\ref{FH}}, and that $\gamma = 1/\log(2/\delta)$. Then for $\delta \in\left]0,2/e\right]$ and for every $k \in \bbN$, with probability at least $1-\delta$ it holds that
\begin{equation}
\label{eq:avg_unbounded_hp_fh}
f(\bar{x}_k)  - f_*
\lesssim \frac{1}{k^{(p-1)/p}} \bigg(1+\log\bigg(\frac{2}{\delta}\bigg)\bigg).
\end{equation}
\end{enumerate}
\end{restatable}
\begin{remark}
Notice that the above convergence rates in high probability are only (up to constants) a $\log(1/\delta)$ factor worse than their in expectation counterpart. Moreover, similarly to the case of the convergence in expectation, our results improves the results in \cite{Liu2023b} from  $(\log^2 k)/k^{(p-1)/p}$ to $(\log^{1/p} k)/k^{(p-1)/p}$.
\end{remark}

\section{Analysis of the last iterate}
\label{sec:LI}
Differently from the analysis of the average iterate, for the last iterate we present two different analysis of the convergence in expectation depending on whether the horizon $k$ is known before running the algorithm (and can then be used to tune the parameters) or not.

We start from the following basic lemma that will serve both finite horizon and infinite horizon settings.
\begin{lemma}
\label{lm:li_basic_lemma_fh_unbounded}
Considering the sequence $(x_k)_{k \in \N}$ generated by Algorithm~\ref{algorithm:clippedSsGM},
let $k \in \bbN$ be fixed, and pick $1 \leq i_0 \leq i_1 \leq k$. Then the following holds.
\begin{equation}
\label{eq:li_basic_lemma_fh_unbounded}
\sum_{i=i_0}^{i_1} \gamma_i \bbE[(f(x_i)-f(x_{i_0}))] \leq a_p \sum_{i=i_0}^{i_1} \frac{\gamma_i}{\lambda_i^{p-1}} \bbE \norm{x_i-x_{i_0}} + \frac{b_p}{2} \sum_{i=i_0}^{i_1} \gamma_i^2 \lambda_i^{2-p}.
\end{equation}
\end{lemma}
\begin{proof}
Again from the standard inequality \eqref{eq:standardineq} with $x=x_{i_0}$ we have 
\begin{align*}
\gamma_i(f(x_i)-f(x_{i_0})) &\leq \frac{1}{2} \norm{x_i-x_{i_0}}^2 - \frac{1}{2} \norm{x_{i+1}-x_{i_0}}^2 + \gamma_i \scalarp{\tilde{\ug}_i-\ug_i, x_{i_0}-x_i} + \frac{1}{2} \gamma_i^2 \norm{\tilde{\ug}_i}^2.
\end{align*}
Summing all the above inequalities from $i_0$ to $i_1$ and taking the expectation of both sides, one obtains
\begin{align*}
\sum_{i=i_0}^{i_1} \gamma_i \bbE[(f(x_i) - f(x_{i_0}))] &\leq \frac{1}{2} \sum_{i=i_0}^{i_1} \bbE[\norm{x_i-x_{i_0}}^2 - \norm{x_{i+1}-x_{i_0}}^2]\\ 
&\qquad+ \sum_{i=i_0}^{i_1} \gamma_i \bbE \big[\bbE_{i_0}\scalarp{\tilde{\ug}_i-\ug_i, x_{i_0}-x_i} \big]
+ \frac{1}{2} \sum_{i=i_0}^{i_1} \gamma_i^2 \bbE \norm{\tilde{\ug}_i}^2. 
\end{align*}
The first sum is bounded from above by 0 as a result of the telescoping. Using \Cref{lm:useful_bounds} with $\xx=x_{i_0}$ on the second and the third sums leads to
\begin{align*}
\sum_{i=i_0}^{i_1} \gamma_i \bbE[(f(x_i)-f(x_{i_0}))] &\leq a_p \sum_{i=i_0}^{i_1} \frac{\gamma_i}{\lambda_i^{p-1}} \bbE \norm{x_i-x_{i_0}} + \frac{1}{2} b_p \sum_{i=i_0}^{i_1} \gamma_i^2 \lambda_i^{2-p}.    
\end{align*}
The statement follows.
\end{proof}

\subsection{Finite-horizon setting}
The study of the convergence of the last iterate relies on reducing the related error to that of the 
the average iterate, which is already known to be optimal. This reduction involves partitioning the execution time into \emph{epochs}: the behavior of the algorithm in the first epoch mirrors that of the average iterate, while the errors in the subsequent epochs can be related to one another and collectively contribute with an overall term of order $1/k^{(p-1)/p}$.

In this section, we consider the policy \ref{FH2} for the stepsize and the clipping level and
 start with a result taken from \cite{Jain2021}, reporting the proof in Appendix~\ref{app:B} for completeness.
\begin{restatable}{lemma}{keyepochlenght}
\label{lm:epoch_length}
Let  $(k_j)_{0 \leq j\leq n-1}$ be the sequence defined in {\rm\ref{FH2}}. Then,
\begin{equation*}
\forall\, j \in \{0,\dots,n-1\}\colon\ 4(k_{j+2}-k_{j+1}) \geq (k_{j+1}-k_j).
\end{equation*}
\end{restatable}
%

The next lemma relates the best iterate values between two successive epochs and is the key element enabling a reduction to the average iterate error. 
%
\begin{restatable}{lemma}{epochlink}
\label{lm:epoch_link_unbounded}
Referring to Algorithm~\ref{algorithm:clippedSsGM}, let $k \in \N$
and suppose that $(\gamma_i)_{1 \leq i \leq k}$ and $(\lambda_i)_{1 \leq i \leq k}$ are chosen as in {\rm\ref{FH2}}. 
Let $\min\bar{f}_j= \min_{i \in E_j} \bbE[f(x_i)]$ be the minimum value in expectation of the function $f$ within the $j$-th epoch.
Then 
for every $j \in \{0,1,\dots,n-1\}$
\begin{align}
\label{eq:epoch_link_unbounded}
\min\bar{f}_{j+1}- \min\bar{f}_j
\leq \frac{5}{2^{(p-1)j} k^{(p-1)/p}} \Bigg[ \norm{x_1-\xx_*} v_p + \gamma
\bigg( \frac{v_p}{\sqrt{k}}+\sqrt{w_p} \bigg)^2 
 \Bigg],
\end{align}
where $v_p = 4 a_p/\lambda^{p-1}$, $w_p = b_p \max\{\Le, \lambda \}^{2-p}$.
\end{restatable}
\begin{proof}
Let $j \leq n-1$. Suppose that
$\min\bar{f}_{j+1} > \min\bar{f}_j$, 
otherwise there is nothing to prove. Let
\begin{equation*}
\tau(j) \in \argmin_{i \in E_j} \bbE [f(x_i)]\quad\text{and}\quad
\tau(j+1) \in \argmin_{i \in E_{j+1}} \bbE [f(x_i)].
\end{equation*}
Then, equation \eqref{eq:li_basic_lemma_fh_unbounded} in \Cref{lm:li_basic_lemma_fh_unbounded},
with $i_0=\tau(j)$ and $i_1=k_{j+2}$, divided by $k_{j+2}-\tau(j)+1$,yields
\begin{align}
\label{eq:20250327e}
\nonumber\sum_{i=\tau(j)}^{k_{j+2}} \gamma_i &\frac{\bbE[f(x_i)-f(x_{\tau(j)})]}{k_{j+2}-\tau(j)+1}\\ 
\nonumber&\leq \frac{a_p \sum_{i=\tau(j)}^{k_{j+2}} (\gamma_i/\lambda_i^{p-1}) \bbE \norm{x_i-x_{\tau(j)}} + (b_p/2) \sum_{i=\tau(j)}^{k_{j+2}} \gamma_i^2 \lambda_i^{2-p}}{k_{j+2}-\tau(j)+1} \\
&\leq \frac{\gamma}{2^{pj} k} \Bigg( \frac{1}{k_{j+2}-\tau(j)+1} \frac{a_p}{\lambda^{p-1}} \sum_{i=\tau(j)}^{k_{j+2}} \bbE \norm{x_i-x_{\tau(j)}} + \frac{b_p}{2} \max\left\{\Le, \lambda \right\}^{2-p}  \gamma \Bigg),
\end{align}
where in the last inequality we used that, for every $i \in E_j \cup E_{j+1}$, it holds
\begin{align*}
\frac{\gamma_i}{\lambda_i^{p-1}} &\leq \frac{\gamma}{\lambda^{p-1}} \frac{1}{2^{jp}k}, \\
\gamma_i^2 \lambda_i^{2-p} &\leq \gamma^2 \max\{\lambda, \Le\}^{2-p} \frac{1}{2^{jp}k}.
\end{align*}
Noting that $\bbE[f(x_i)] \geq \bbE[f(x_{\tau(j)})]$ for every $i \in E_j$, recalling the definition of $\tau(j+1)$, and making use of \Cref{lm:epoch_length}, the left-hand side of \eqref{eq:20250327e}
can be bounded from below as follow
\begin{align*}
\sum_{i=\tau(j)}^{k_{j+2}} \gamma_i \frac{\bbE[f(x_i)-f(x_{\tau(j)})]}{k_{j+2}-\tau(j)+1}
&\geq \sum_{i=k_{j+1}+1}^{k_{j+2}} \gamma_i \frac{\bbE[f(x_i)-f(x_{\tau(j)})]}{k_{j+2}-\tau(j)+1} \\
&\geq \frac{\gamma}{2^{j+1} k^{1/p}} \bbE[f(x_{\tau(j+1)})-f(x_{\tau(j)})] \frac{k_{j+2}-k_{j+1}}{k_{j+2}-k_j} \\
&= \frac{\gamma}{2^{j+1} k^{1/p}} \bbE[f(x_{\tau(j+1)})-f(x_{\tau(j)})] \frac{1}{1 + \frac{k_{j+1}-k_j}{k_{j+2}-k_{j+1}}} \\
&\geq \frac{\gamma}{5\cdot 2^{j+1} k^{1/p}} \bbE[f(x_{\tau(j+1)})-f(x_{\tau(j)})].
\end{align*}
Thus, combining the above inequality with \eqref{eq:20250327e} and recalling the definitions of $v_p$ and $w_p$, we obtain
\begin{align}
\label{eq:link_aux_0}
\nonumber
\bbE&[f(x_{\tau(j+1)})-f(x_{\tau(j)})]\\
&\leq \frac{5}{2^{(p-1)j} k^{(p-1)/p}} \Bigg( \frac{1}{k_{j+2}-\tau(j)+1} \cdot \frac{v_p}{2} \underbrace{\sum_{i=\tau(j)}^{k_{j+2}} \bbE \norm{x_i-x_{\tau(j)}}}_{(\star)} + w_p  \gamma \Bigg).
\end{align}
Now, we bound the term $(\star)$ above. To that purpose
we observe that the quantities $A_k$ and $B_k$ defined in Theorem~\ref{thm:avg_unbounded},
with the parameters' definitions considered in \ref{FH2} can be bounded as follows
\begin{align}
\label{eq:li_fh_important_bounds1}
A_k=a_p\sum_{i=1}^k \frac{\gamma_i}{\lambda_i^{p-1}} &= a_p\sum_{j=0}^n \sum_{i \in E_j} \frac{\gamma_i}{\lambda_i^{p-1}} \leq a_p \frac{\gamma}{k \lambda^{p-1}} \sum_{j=0}^{n} \bigg(\frac{1}{2^p} \bigg)^j \leq \frac{2 \gamma a_p}{k \lambda^{p-1}} = \frac{\gamma v_p}{2k}
\end{align}
and
\begin{align}
\label{eq:li_fh_important_bounds2}
\nonumber B_k &= b_p\sum_{i=1}^k \gamma_i^2 \lambda_i^{2-p} = b_p\sum_{j=0}^n \sum_{i \in E_j} \gamma_i^2 \lambda_i^{2-p}\\
& \leq b_p \frac{\gamma^2 \max\{\lambda, \Le\}^{2-p}}{k} \sum_{j=0}^n \bigg(\frac{1}{2^p}\bigg)^j \leq b_p\frac{2 \gamma^2 \max\{\lambda, \Le\}^{2-p}}{k} = \frac{2\gamma^2 w_p}{k}.
\end{align}
Thus, the triangle inequality, and inequality \ref{eq:distance_bound} in Theorem~\ref{thm:avg_unbounded} combined with inequalities \eqref{eq:li_fh_important_bounds1}
and \eqref{eq:li_fh_important_bounds2}, yield
that for every $i \in \{\tau(j), \dots, k_{j+2}\}$,
\begin{align*}
\bbE \norm{x_i-x_{\tau(j)}} &\leq \bbE \norm{x_i-\xx_*} + \bbE \norm{x_{\tau(j)}-\xx_*}\\ 
&\leq 2 \norm{x_{1}-\xx_*} + 2\Bigg(\frac{v_p}{k} + \sqrt{\frac{2 w_p}{k}} \Bigg) \gamma.
\end{align*}
Plugging the above inequality into \eqref{eq:link_aux_0} and rearranging leads to
\begin{align*}
\bbE[f(x_{\tau(j+1)})-f(x_{\tau(j)})] &\leq \frac{5}{2^{(p-1)j} k^{(p-1)/p}} \Bigg[ \norm{x_1-\xx_*} v_p + \Bigg( \frac{v^2_p}{k} + v_p \sqrt{\frac{2 w_p}{k} } + w_p  \Bigg) \gamma \Bigg]
\end{align*}
and the statement follows.    
\end{proof}

We now ready to give the main result of this section.
\begin{theorem}
\label{thm:li_unbounded_fh}
Referring to Algorithm~\ref{algorithm:clippedSsGM}, let $k \in \N$
and suppose that $(\gamma_i)_{1 \leq i \leq k}$ and $(\lambda_i)_{1 \leq i \leq k}$ are chosen as in {\rm\ref{FH2}}. Let $\xx_*\in \argmin_{\XX}\hspace{-0.1ex} f$ and $f_* = \min_{\XX} f$. Then, for every $k \in \N$,
\begin{align}
\label{eq:li_unbounded_fh}
\bbE[f(x_k)] -f_* &
\leq \frac{1}{k^{(p-1)/p}} \Bigg[  \frac{6}{5} \frac{\norm{x_{1}-\xx_*}^2}{\gamma} + \frac{53}{10}\cdot\frac{\cdot 2^{p-1}}{2^{p-1}-1} \Big(\norm{x_1-\xx_*} v_p + \gamma 
\big( v_p+\sqrt{w_p} \big)^2 
\Big)\Bigg]
\end{align}
where $v_p$ and $w_p$ are defined as in Lemma~\ref{lm:epoch_link_unbounded}.
\end{theorem}
\begin{proof}
Relying on the notation of \Cref{lm:epoch_link_unbounded} and noting that $E_n = \{k\}$, we have
\begin{align*}
\bbE[f(x_k)] =
\min \bar{f}_{n} &= \min \bar{f}_{0}  + \sum_{j=0}^{n-1} \big(\min \bar{f}_{j+1}- \min \bar{f}_{j} \big)\\
&=\min_{i \in E_0} \bbE[f(x_i)]  + \sum_{j=0}^{n-1} \big(\min \bar{f}_{j+1}- \min \bar{f}_{j} \big).
\end{align*}
Using \Cref{lm:epoch_link_unbounded} and the fact that $\sum_{j=0}^{n-1} (1/2^{p-1})^j \leq 2^{p-1}/(2^{p-1}-1)$, one obtains
\begin{align*}
\bbE[f(x_k)] &\leq \min_{i \in E_0} \bbE[f(x_i)] + \frac{2^{p-1}}{2^{p-1}-1} \frac{5}{k^{(p-1)/p}} \Bigg[ \norm{x_1-\xx_*} v_p + 
\gamma \bigg( \frac{v_p}{\sqrt{k}}+\sqrt{w_p} \bigg)^2 
\Bigg].
\end{align*}
Noting that the minimum is always smaller than the average and subtracting $f_*$ from both sides, we have
\begin{align*}
\bbE[f(x_k)-f_*] &\leq \frac{1}{k_1} \sum_{i=1}^{k_1} \bbE[f(x_{i})-f_*] 
+ \frac{2^{p-1}}{2^{p-1}-1} \frac{5}{k^{(p-1)/p}} 
\Bigg[  \norm{x_1-\xx_*} v_p + 
\gamma \bigg( \frac{v_p}{\sqrt{k}}+\sqrt{w_p} \bigg)^2 
\Bigg].
\end{align*}
The first term in the right-hand side of the above inequality can be bounded using the bound in \ref{eq:avg_unbounded} of \Cref{thm:avg_unbounded}, so obtaining 
\begin{align*}
\bbE[f(x_k)-f_*] &\leq \frac{1}{k_1 \gamma_{k_1}} \bigg[ \frac{5}{9}\big(\norm{x_{1}-\xx_*}^2+ B_{k_1}\big) + 2 A^2_{k_1} \bigg] \\
&\qquad\qquad+ \frac{2^{p-1}}{2^{p-1}-1} \frac{5}{k^{(p-1)/p}} 
\Bigg[  \norm{x_1-\xx_*} v_p + 
\gamma \bigg( \frac{v_p}{\sqrt{k}}+\sqrt{w_p} \bigg)^2 
\Bigg],
\end{align*}
where we recall that $A_{k_1} = a_p \sum_{i=1}^{k_1} \gamma_i/\lambda_i^{p-1}$ and
$B_{k_1} = b_p \sum_{i=1}^{k_1} \gamma_i^2 \lambda_i^{2-p}$.
Now noting that $k/3 \leq k_1 \leq k/2$ and that, for every $i \in E_0$, $\gamma_i = \gamma/k^{1/p}$ and $\lambda_i = \max\{\lambda k^{1/p}, \Le \}$, we have
\begin{align*}
\bbE[f(x_k)-f_*] &\leq \frac{3}{\sqrt{2}}\cdot\frac{1}{k^{(p-1)/p}} \Bigg[ \frac{5}{9}\frac{\norm{x_{1}-\xx_*}^2}{\gamma} +  \frac \gamma 2\bigg( \frac{1}{16}v_p^2 + \frac{5}{9}w_p\bigg)  \Bigg] \\
&\qquad+ \frac{2^{p-1}}{2^{p-1}-1} \frac{5}{k^{(p-1)/p}} 
\Bigg[ \norm{x_1-\xx_*} v_p + 
\gamma \bigg( \frac{v_p}{\sqrt{k}}+\sqrt{w_p} \bigg)^2 
\Bigg].
\end{align*}
Rearranging the terms we have
\begin{align*}
\bbE[f(x_k)-f_*] &\leq \frac{1}{k^{(p-1)/p}} \Bigg[ \frac{5 \cdot 2^{p-1}}{2^{p-1}-1} \norm{x_1-\xx_*} v_p + \frac{5}{3\sqrt{2}}\cdot \frac{\norm{x_{1}-\xx_*}^2}{\gamma} \\
&\qquad\qquad\qquad+ \gamma\Bigg(\frac{5}{6\sqrt{2}}\big( v_p^2 + w_p\big) +\frac{5\cdot 2^{p-1}}{2^{p-1}-1} \bigg( \frac{v_p}{\sqrt{k}}+\sqrt{w_p} \bigg)^2  \Bigg) \Bigg].
\end{align*}
The statement follows by noting that
\begin{equation*}
\frac{5}{6\sqrt{2}}\big( v_p^2 + w_p\big)\leq
\frac{5}{12\sqrt{2}}\cdot 2 \big( v_p + \sqrt{w_p}\big)^2
\leq \frac{5}{12\sqrt{2}}\frac{2^{p-1}}{2^{p-1}-1} 
\big( v_p + \sqrt{w_p}\big)^2 
\end{equation*}
\end{proof}
\begin{remark}
The following comments are in order.
\begin{itemize}
\item The rate in \eqref{eq:li_unbounded_fh} features the same order of $1/k^{(p-1)/p}$ for the average iterate in the finite horizon case and it is therefore optimal. 

\item Examining the bound in \Cref{thm:li_unbounded_fh}, we see that the convergence rate for the last iterate exhibits substantially worse constants than the analogous rate for the average iterate in \eqref{eq:avg_unbounded_fh}. In particular, the term $2^{p-1}/(2^{p-1}-1)$ grows rapidly and without bound as $p$ approaches 1. It remains unclear whether this constant can be improved.
\end{itemize}
\end{remark}

\subsection{Anytime setting}
Similarly to the finite horizon case, we reduced the analysis of the error of the $k$-th iterate to that of a weighted (by the step size) average. In this setting the stepsizes and clipping levels are set as is {\rm \ref{AT}}. 
Thus, we have
\begin{equation}
\label{eq:boundAT}
\begin{aligned}
\frac{\gamma_i}{\lambda_i^{p-1}} 
&\leq \frac{\gamma}{\lambda^{p-1}} \frac{1}{i(1+\log i)}\\
\gamma_i^2 \lambda_i^{2-p} 
&\leq  \gamma^2 b_p \max\{ \Le, \lambda\}^{2-p} \frac{1}{i(1+\log i)}
\end{aligned}
\end{equation}
and hence, by elementary computation (see Lemma~\ref{lm:technical_lemma}),
\begin{equation}
\label{eq:20250329a}
\begin{aligned}
A_k &\leq \frac{\gamma a_p}{\lambda^{p-1}} \sum_{i=1}^k \frac{1}{i(1+\log i)} \leq 
\frac{\gamma a_p}{\lambda^{p-1}} (1+\log(1+\log k))\\
B_k & \leq \gamma^2 b_p \max\{\Le, \lambda\}^{2-p}\sum_{i=1}^k \frac{1}{i(1+\log i)} 
\leq \gamma^2 b_p \max\{\Le, \lambda\}^{2-p} (1+\log(1+\log k)).
\end{aligned}
\end{equation}

%
The following results appeared in \cite{Lin2018} (for completeness, we give a proof in Appendix~\ref{app:B}) and enables the reduction to the weighted average of the errors.
\begin{restatable}{lemma}{keyinequalityat}
\label{lm:key_inquelity_at}
Let $(\beta_i)_{1 \leq i \leq k} \in \R^k$. Then the following hold.
\begin{enumerate}[label={\rm(\roman*)}]
\item\label{lm:key_inquelity_at_i} $\displaystyle \beta_k = \frac{1}{k} \sum_{i=1}^k \beta_i + \sum_{j=1}^{k-1} \frac{1}{j(j+1)}\sum_{i=k-j+1}^k (\beta_i - \beta_{k-j})$.
\item\label{lm:key_inquelity_at_iii} If $(i\beta_i)_{1 \leq i\leq k}$ is decreasing, then 
$\displaystyle \sum_{j=1}^{k-1} \frac{1}{j(j+1)}\sum_{i=k-j}^k \beta_i \leq \frac{2}{k}\sum_{i=1}^k \beta_i + \beta_k\bigg(1-\frac{3}{k}\bigg)$.
\end{enumerate}
\end{restatable}
%
%
\begin{theorem}
\label{thm:li_unbounded_at}
Referring to Algorithm~\ref{algorithm:clippedSsGM},
 suppose that $(\gamma_k)_{k \in \N}$ and $(\lambda_k)_{k \in \N}$ are determined as in {\rm \ref{AT}}.
 Let $\xx_*\in \argmin_{\XX}\hspace{-0.1ex} f$ and $f_* = \min_{\XX} f$. Then, for every $k \in \N$,
\begin{align}
\label{eq:li_unbounded_at}
\nonumber\bbE[f(x_k)] -f_* &\leq \frac{(1+\log k)^{1/p}}{k^{(p-1)/p}}
\Bigg[\frac 5 9 \cdot \frac{\norm{x_{1}-\xx_*}^2}{\gamma} 
+ \frac{6}{5}
v_p \norm{x_1- \xx_*}\\
&\qquad\qquad\qquad\qquad\qquad 
+ \frac{3}{4}
\gamma \big(1+\log (1+\log k)\big)^2\big( v_p + \sqrt{w_p} \big)^{\!2}\, \Bigg].
\end{align}
\end{theorem}
\begin{proof}
We start by applying \Cref{lm:key_inquelity_at}\ref{lm:key_inquelity_at_i} with 
$\beta_i = \gamma_i \bbE[f(x_i)-f_*]$,
obtaining 
\begin{align*}
\gamma_k \bbE[f(x_k)-f_*] &
\leq \frac{1}{k} \sum_{i=1}^{k} \gamma_i \bbE[f(x_i)-f_*] \\
&\quad\qquad+ \sum_{j=1}^{k-1}\frac{1}{j(j+1)} \sum_{i=k-j+1}^k \big(\gamma_i \bbE[f(x_i)-f_*] - \gamma_{k-j}\bbE[f(x_{k-j})-f_*]\big) \\
&\leq \underbrace{\frac{1}{k} \sum_{i=1}^{k} \gamma_i \bbE[f(x_i)-f_*]}_{\textbf{A}} + \underbrace{\sum_{j=1}^{k-1}\frac{1}{j(j+1)} \sum_{i=k-j}^k \gamma_i (\bbE[f(x_i)-f(x_{k-j})])}_{\textbf{B}},
\end{align*}
where in the last inequality we used the fact that $(-\gamma_i)_{1\leq i\leq k}$ is increasing.
In the following we set for the sake of brevity 
\begin{equation*}
\eta_i = i(1+\log i)\quad\text{and}\quad \alpha_k = \sum_{i=1}^k \frac{1}{\eta_i} = 1+ \log(1+\log k).
\end{equation*}
Then, it is immediate from \eqref{eq:thm_avg_unbounded_p1} in the proof of \Cref{thm:avg_unbounded} that 
\begin{align*}
\textbf{A} &\leq \frac{1}{k} \Big(\frac{5}{9}\big(\norm{x_{1}-\xx_*}^2+B_k\big) + 2 A_k^2  \Big),
\end{align*}
and hence, recalling \eqref{eq:20250329a} and the definition of $v_p$ and $w_p$ in Lemma~\ref{lm:epoch_link_unbounded},
\begin{align}
\label{eq:A_bound}
\textbf{A} \leq \frac{1}{k} \bigg[\frac 5 9 \norm{x_{1}-\xx_*}^2 +  
\gamma^2\bigg(\frac{v^2_p}{8} \alpha_k +  \frac 5 9 w_p  \bigg)\alpha_k
\bigg].
\end{align}

Concerning the term \textbf{B}, we first note that by Lemma~\ref{lm:li_basic_lemma_fh_unbounded}
with $i_0 = k-j$ and $i_1 = k$ and the bounds in \eqref{eq:boundAT}, we have
%
\begin{align}
\label{eq:20250403d}
\nonumber
\sum_{i=k-j}^k \gamma_i  \bbE[f(x_i) - f(x_{k-j})] 
&\leq a_p \sum_{i=k-j}^k \frac{\gamma_i}{\lambda_i^{p-1}} \bbE \norm{x_i-x_{k-j}} + \frac{b_p}{2} \sum_{i=k-j}^k \gamma_i^2 \lambda_i^{2-p}\\
&\leq 
\frac{\gamma v_p}{4}
\sum_{i=k-j}^k \frac{\bbE \norm{x_i-x_{k-j}}}{\eta_i}+ 
\frac{\gamma^2 w_p}{2}
\sum_{i=k-j}^k \frac{1}{\eta_i}.
\end{align}
Now notice that, the triangle inequality combined with Theorem~\ref{thm:avg_unbounded}\ref{eq:distance_bound}
and \eqref{eq:20250329a}
 yields
\begin{align*}
\bbE \norm{x_i - x_{k-j}} &\leq \bbE\norm{x_i - \xx_*} + \bbE\norm{x_{k-j}-\xx_*}\\
&\leq 2 \Big( \norm{x_1-\xx_*} +  2 A_k + \sqrt{B_k } \Big)\\
&\leq 2\norm{x_1-\xx_*} +  2\gamma \Big( \frac{v_p}{2} \alpha_k + \sqrt{w_p \alpha_k } \Big),
\end{align*}
and hence, since $(1/\eta_i)_{1\leq i \leq k}$ is decreasing, thanks to Lemma~\ref{lm:key_inquelity_at}\ref{lm:key_inquelity_at_iii}, we have
\begin{align}
\label{eq:B_bound}
\nonumber
\textbf{B} &\leq \Bigg( \frac{1}{2} \gamma v_p  \norm{x_1-\xx_*} + \gamma^2 \frac{v_p}{2}\Big( \frac{v_p}{2} \alpha_k + \sqrt{w_p \alpha_k } \Big) + \frac{w_p \gamma^2}{2}  \Bigg)
\sum_{j=1}^{k-1}\frac{1}{j(j+1)} \sum_{i=k-j}^k \frac{1}{\eta_i}\\
\nonumber& \leq \Bigg( \frac{1}{2}\gamma v_p \norm{x_1-\xx_*} + \frac{\gamma^2}{2}\bigg(\frac 12 v^2_p\alpha_k + v_p\sqrt{w_p\alpha_k}\bigg)   \Bigg) \bigg( \frac{2}{k}\sum_{i=1}^k \frac{1}{\eta_i} + \frac{1}{\eta_k}\bigg(1-\frac{3}{k}\bigg)\bigg)\\
&\leq \frac{21}{20}
\bigg(  \gamma v_p\norm{x_1-\xx_*} +  \gamma^2 \bigg(\frac 12 v^2_p\alpha_k + v_p\sqrt{w_p\alpha_k}\bigg) \bigg)  \frac{\alpha_k}{k},
\end{align}
where in the last inequality we used the fact that
$(k-3)/\eta_k \leq 10^{-1}\cdot\alpha_k$.
Combining \eqref{eq:A_bound} and \eqref{eq:B_bound}, we have
\begin{align*}
\gamma_k  \bbE[f(x_k)-f_*] &\leq 
\frac{1}{k} \bigg[\frac 5 9 \norm{x_{1}-\xx_*}^2 
+ \frac{21}{20}
\gamma v_p\norm{x_1- \xx_*}\\
&\qquad\qquad+ \gamma^2 \alpha_k\bigg( \frac{1}{8}v^2_p\alpha_k + \frac 5 9 w_p 
+ \frac{21}{40} 
v_p^2 \alpha_k 
+ \frac{21}{20} 
v_p\sqrt{\alpha_k w_p}\bigg)
\bigg]\\
&\leq 
\frac{1}{k} \bigg[\frac 5 9 \norm{x_{1}-\xx_*}^2 
+ \frac{21}{20} 
\gamma v_p\norm{x_1- \xx_*}
+ \frac{13}{20}
\gamma^2 \alpha_k\big( v_p \sqrt{\alpha_k} + w_p \big)^2
\bigg].
\end{align*}
The statement follows by dividing both sides by $\gamma_k$, collecting $\gamma$, 
and majorizing the fractions.
\end{proof}

\begin{remark} 
The rate in \eqref{eq:li_unbounded_at} is (ignoring polylogarithmic factors in $\log k$) of order  $(\log^{1/p} k)/k^{(p-1)/p}$, and we do not know if the $\log^{1/p} k$ factor can be improved. For comparison, we notice that in the special case $p=2$, our rate is of order  $\sqrt{(\log k)/k}$ and improves upon the result in \cite{Harvey2019a} by shaving a $\sqrt{\log k}$ factor. 
\end{remark}

\section{Application to Supervised Learning with Kernels}
\label{sec:SL}
In this section, we consider the \emph{supervised learning} problem defined below.
\begin{align}
\label{eq:statistical_learning}
\minimize{\xx \in \HH}{R(\xx)= \bbE[\ell(\scalarp{\xx,\phi(Z)},Y)]},
\end{align}
under the following assumptions:
\begin{enumerate}[label=A\arabic*{\rm}]
\item\label{A1} $Z$ and $Y$ are random variables, taking values in the abstract measure spaces $\cZ$ and $\cY$ respectively, having joint \emph{distribution} $\mu$.

\item\label{A2} $\HH$ is a real separable Hilbert space with inner product $\scalarp{\cdot,\cdot}$ and $\phi\colon \cZ \rightarrow \cX$ is a \emph{feature map} with a corresponding \emph{kernel} $K(\cdot, \cdot) = \scalarp{\phi(\cdot), \phi(\cdot)}$.

\item\label{A3} For $p \in \left]1,2\right]$, there exists $\nu > 0$ such that $\bbE [K(Z,Z)^{p/2}] = \bbE [\norm{\phi(Z)}^p] \leq \nu^p$.

\item\label{A4} $\ell\colon \bbR\times  \cY  \rightarrow \bbR_+$ is a \emph{loss function} convex and Lipschitz continuous in the first argument, with Lipschitz constant $\cL > 0$ and satisfying $\bbE[\ell(0,Y)] < \infty$.

\item\label{A5} $\argmin_{\xx \in \HH} R(\xx)\neq \varnothing$.
\end{enumerate}
In words, the goal is to find a \emph{prediction function} $f_\xx\colon \cZ\to \R\colon z \mapsto \scalarp{\xx, \phi(z)}$, parametrized by $\xx \in \HH$, with the smallest expected loss, by accessing only to a finite sample of i.i.d. copies of $(Z,Y)$.
%

Problem \eqref{eq:statistical_learning} is a special case of the problem consider in \eqref{eq:mainprob}. Indeed, it is an easy consequence of the above assumptions
that the hypothesis 
\ref{H1}--\ref{H4} are satisfied with the following stochastic subgradient 
\begin{align}
\label{eq:sl_oracle}
\hat{\ug}(\xx, (Z, Y)) = \ell^{\prime}(\scalarp{\xx, \phi(Z)}, Y) \phi(Z),
\end{align}
and
\begin{equation*}
L = \cL \cdot \bbE [\norm{\phi(Z)}]\quad\text{and}\quad \sigma^p = (2\nu \cL)^p.
\end{equation*}
Details are given in Appendix~\ref{sec:appC}.
%
%
Therefore, we can we derive from \Cref{crl:avg_unbounded_hp}\ref{crl:avg_unbounded_hp_fh}
the following guarantees.
\begin{corollary}
\label{crl:sl_avg_hp}
Consider Algorithm~\ref{algorithm:clippedSsGM} applied to the oracle \eqref{eq:sl_oracle} (see Algorithm~\ref{algo:kernelSsGM}). For $\delta \in (0, 2/e]$, let
\begin{align}
\label{eq:sl_parameter_setting}
(\gamma_i)_{1 \leq i \leq k} \equiv \frac{(\log(2/\delta))^{-1/2}}{k^{1/p}}, (\lambda_i)_{1 \leq i \leq k} \equiv \max\{\Le, k^{1/p}\}.
\end{align}
Then, with probability at least $1-\delta$, it holds
\begin{align}
\label{eq:sl_avg_hp}
R(\bar{x}_k) - R^* \lesssim \frac{1}{k^{(p-1)/p}} \Big(1 + \log\big(1/\delta\big)\Big).
\end{align}
\end{corollary}
\begin{remark}
The following comments are in order.
\begin{enumerate}[label={\rm(\roman*)}]
\item Algorithm~\ref{algorithm:clippedSsGM} can be easily implemented to solve the problem in \eqref{eq:statistical_learning}, even when $\HH$ is infinite dimensional, via the kernel function $K$.
 This has been first noticed in \cite{Parletta2024} that provided high probability bounds for a constrained version of the problem in \eqref{eq:statistical_learning} under the restrictive assumption that the constrain set is bounded. For reader's convenience, we provide full implementation details in Appendix~\ref{sec:appC}.
\item The parameter setting in \eqref{eq:sl_parameter_setting} is a special case of \ref{FH}, where we set $\gamma = (\log(2/\delta))^{-1/2}$ and $\lambda = 1$. In supervised learning, the random sample $S=(Z_i,Y_i)_{1 \leq i \leq k} \sim \mu^{\otimes k}$ is entirely available to the learner, and then considering the finite horizon setting with $m=1$ and $k$ equal to the sample size is natural. In that case, Algorithm~\ref{algorithm:clippedSsGM} makes a \emph{single pass} over the data. 
\item From Corollary~\ref{crl:avg_unbounded_hp}\ref{crl:avg_unbounded_hp_at},
 Theorem~\ref{thm:li_unbounded_fh}, and Theorem~\ref{thm:li_unbounded_at} analogous rates can be derived for the anytime setting and the last iterate (in expectation). We note that the anytime setting  allows to address an online learning scenario where the data 
 arrives continuously in a potentially infinite stream of batches and the predictive model is updated accordingly.

\item When $\mu$ is \emph{light-tailed}  the problem of learning with convex Lipschitz losses and $\phi(\zz)=\zz$ has been studied in \cite{Steinwart2008,Shalev2010,Dellavecchia2021,Dellavecchia2024}. The more general formulation of light tailed distribution assumes that $Z$ satisfies the following sub-Gaussian condition
\begin{align}
\label{eq:sl_light_tails}
\bbE \scalarp{Z,\zz}^r \leq C \cdot \sqrt{\bbE \scalarp{Z,\zz}^2} \cdot r^{r/2}, \quad (\forall \zz \in \HH, \forall r \geq 2),
\end{align}
where $C$ is an absolute numerical constant.

The canonical approach under \eqref{eq:sl_light_tails} is to perform \emph{regularized risk minimization}, i.e. to compute
\begin{align}
\label{eq:sl_rrm}
\hat{x}_{\eta,k} \in \underset{\xx \in \HH}{\argmin} \frac{1}{k} \sum_{i=1}^k \ell(\scalarp{\xx,Z_i}, Y_i) + \eta \norm{\xx}^2,
\end{align}
where $\eta > 0$. The best results for this approach are given in \cite{Dellavecchia2024}, where the authors show that with probability at least $1-\delta$ it holds
\begin{align}
\label{eq:sl_rrm_bound}
R(\hat{x}_{\eta,k}) - R^* \lesssim \frac{1}{\sqrt{k}} \bigg( \log \log k + \sqrt{\log\bigg(\frac{1}{\delta}\bigg)} \bigg),
\end{align}
where $\eta = c \cdot (\log(1/\delta))^{1/2}$ for some numerical constant $c > 0$.

We note that the sub-Gaussian assumption in \eqref{eq:sl_light_tails} implies that $Z$ has finite moments of any order. By contrast, our results ensures that, under the much weaker assumption \ref{A4}, the clipped sub-gradient method with finite horizon $k$ features a rate of convergence in high probability of order 
%
%
\begin{align*}
R(\bar{x}_{k}) - R^* \lesssim \frac{1}{\sqrt{k}} \bigg( 1 + \log\bigg(\frac{1}{\delta}\bigg) \bigg),
\end{align*}
which even improves \eqref{eq:sl_rrm_bound} by a $\log \log k$ factor,
although with a worse dependence on the confidence level. 
\item Alternative approaches are the ones considered in \cite{Brownlees2015,Chinot2020,Geoffrey2020,Lecue2020,Mathieu2021,Oliveira2023}. All these robust methods allow the same order as \eqref{eq:sl_avg_hp} under similar assumptions to clipped SsGM. However, their implementation is significantly more involved and their solution can be only approximated, limiting their practical applicability. 
\end{enumerate}
\end{remark}

\section{Numerical experiments}
\label{sec:experiments}
%
We consider the minimization of the function $f(x)=\|x\|_1$ in $\bbR^{d}$, with $d=100$, under an additive zero-mean component-wise independent Pareto noise with shape parameter 
$\alpha=p+0.001$, with $p\in \left]1,2\right[$, and scale parameter such that $\sigma=1$.
This way the noise has no moments of order strictly greater than $p$. For each run, the initial point $x_1$ is sampled uniformly at random from the surface of the unit sphere in $\bbR^d$. We run each algorithm $1000$ times and report the average value of $f$, we omit the error bars as they are of difficult interpretation due to the use of the $\log$-scale on the $y$-axis. In all experiments, we set the parameter $\varepsilon=0.01$ and the batch-size $m=1$.

\paragraph{{\upshape Comparison with SsGM}} 
The goal of this experiment is to show the advantage of clipping over SsGM in the noise regime described above with $p=1.1$. We consider the finite horizon setting and use the stepsize and clipping level schedules suggested by \Cref{crl:avg_unbounded}\ref{crl:avg_unbounded_fh} and \Cref{thm:li_unbounded_fh}. We consider two variants of SsGM: SsGM with the standard stepsize schedule $\gamma/\sqrt{k}$, and SsGM 2 using the same schedule as C-SsGM. We set the stepsize and clipping level parameters $\gamma$ and $\lambda$ via a grid search. The results, reported in  \Cref{tab:experiment_results}, confirm that the clipping strategy provides a notable acceleration of the convergence for both the average and the last iterate.

\begin{table}[t]
\begin{center}
\begin{tabular}{||c c c c||} 
\hline
Output & SsGM & SsGM2 & C-SsGM \\ [0.5ex] 
\hline\hline
$\bar{x}_k$ & 4.227 & 1.921 & 1.218 \\ 
\hline
$x_k$ & 5.767 & 0.102 & 0.003 \\ 
\hline
\end{tabular}
\end{center}
\vspace{-1ex}
\caption{Comparison of SsGM against its clipped version (C-SsGM). SsGM and SsGM 2 refer to the stepsize schedules of $1/\sqrt{k}$ and $1/k^{1/p}$ respectively. The figures represents the average value of $f$ at the output of algorithms.}
\label{tab:experiment_results}
\end{table}

\begin{figure*}[t!]
\centering
\includegraphics[scale=0.40]{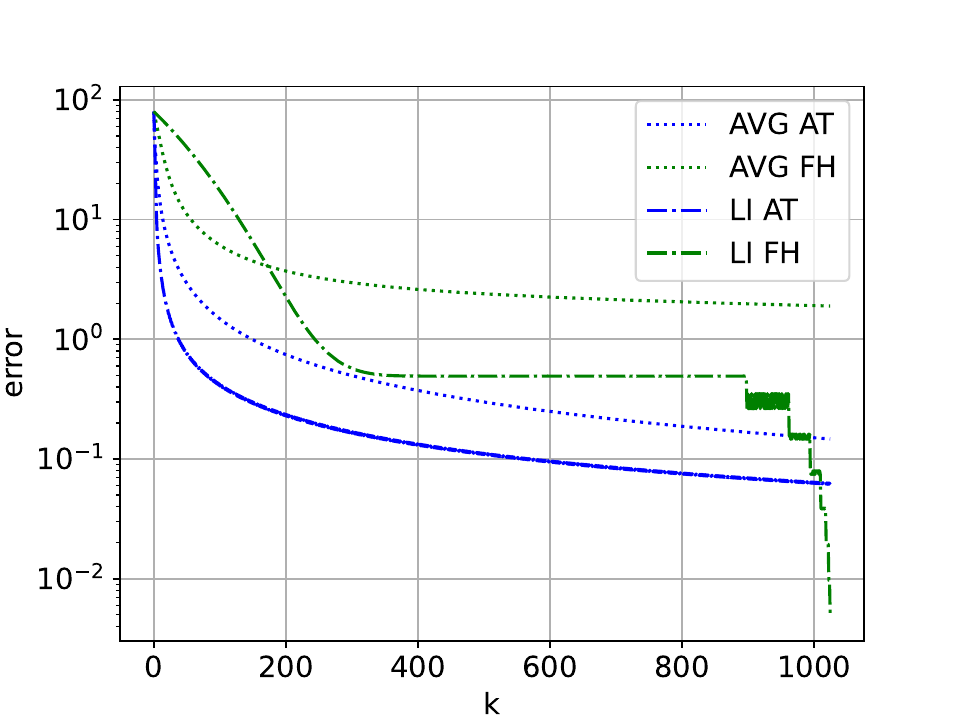}
\includegraphics[scale=0.40]{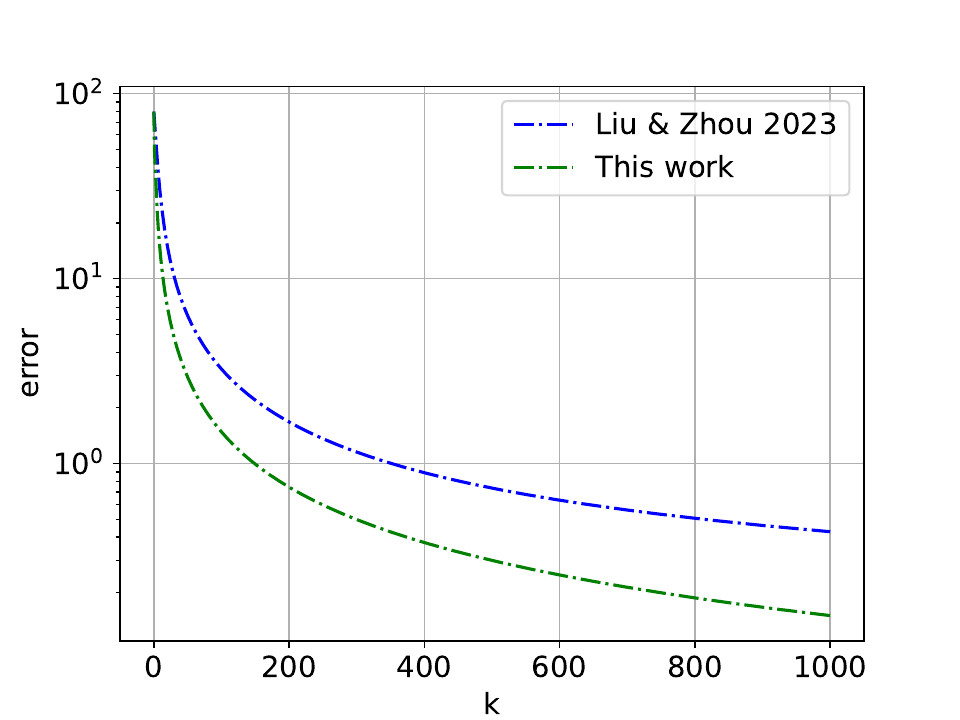}
\caption{Experimental Results. Left panel: comparison of the parameter settings of C-SsgM in the anytime (AT) and the finite horizon (FH) cases. Right panel: comparison between the parameter setting for the average iterate proposed in this work against that proposed in \cite{Liu2023b}.} 
\label{fig:experimental_results}
\end{figure*}

\paragraph{\upshape anytime vs. Finite-horizon} 
In this case we compare the parameter schedules for finite horizon and anytime settings. Results are showed in the left panel of \Cref{fig:experimental_results}. To simplify the plot we report the average performance along the entire time window of each method, but formally, the performance in the finite horizon setting should be only evaluated at the end of the execution time. For the last iterate, the finite horizon policy outperforms the anytime one, as predicted by the theory. 
%
\paragraph{\upshape Comparison with \cite{Liu2023b}} 
We compare the performance of the average iterate with the parameter setting of \Cref{crl:avg_unbounded}\ref{crl:avg_unbounded_at} against the setting of \cite[Theorem 4]{Liu2023b} in the anytime case. The results are showed in the right panel of \Cref{fig:experimental_results}. We notice that the parameter setting proposed in this paper outperforms that in \cite{Liu2023b} as hinted by our theory.

\section{Conclusions}
\label{sec:conclusion}
We considered nonsmooth convex problems and provided the first rates of convergence in expectation for the last iterate of the clipped SsGM algorithm. Our results optimally extend the convergence theory of SsGM for the canonical case of bounded variance to the heavy tails setting. We also improve the rate of the average iterate by shaving an exogenous $\log^{(2p-1)/p}\! k$ factor from the state-of-the-art. Whether the remaining $\log^{1/p}\! k$ is necessary or not, is an interesting open question. Another question is related to the possibility of a unified convergence analysis of the last iterate resulting in optimal rates in both the anytime and the finite horizon setting. Finally, remains 
open the question of rate in high probability for the last iterate under heavy-tailed noise.


\newpage

\begin{appendices}

In this part we provide full proofs of some results presented in the main paper
as well as additional technical facts used in the convergence analysis. Finally, we show how C-SsGM can be effectively implemented for the problem of supervised learning with kernels.

\section{Analysis of the average iterate}
\label{app:A}

We first provide a general result concerning the statistical properties of the clipped subgradient estimator.
Let $\xx \in \HH$ and $\lambda>0$, and let $(\xi_j)_{1\leq j \leq m}$ be a finite sequence
of independent copies of $\xi$. We set
\begin{equation*}
\uu = \bbE[\hat{\ug}(\xx,\xi)],\quad \bar{\ug} = \frac{1}{m} \sum_{j=1}^m \hat{\ug}(\xx,\xi_j),\quad
\tilde{\ug} = \textsc{CLIP}(\bar{\ug}, \lambda).
\end{equation*}
Notice that because of assumption \ref{H4}, there exists $\sigma_m>0$ such that
\begin{equation*}
\bbE[\norm{\bar{\ug}- \uu}^p] \leq \sigma_m^p.
\end{equation*}
For instance, when $p=2$, one can take $\sigma_m = \sigma/\sqrt{m}$.

\begin{lemma}
\label{lm:stat_prop}
Under the assumptions~{\rm\ref{H1}}--{\rm\ref{H4}} and the above notation, the following hold.
\begin{enumerate}[label={\rm(\roman*)}]
\item\label{lm:stat_prop_moment}
For $r\geq p$ we have $\bbE[\|\tilde{\ug}\|^r] \leq b_p \lambda^{r-p}$ \quad {\rm ($r$-th moment)},
\item\label{lm:stat_prop_bias}
$\norm{\bbE\tilde{\ug} - \uu} \leq a_p \cdot \lambda^{1-p}$ \quad {\rm (Bias)},
\item\label{lm:stat_prop_var}
$\bbE[\norm{\tilde{\ug}- \bbE \tilde{\ug}}^2] \leq b_p \cdot \lambda^{2-p}$
\end{enumerate}
where we recall {\rm(}see definition \eqref{eq:constants}{\rm)} that 
\begin{equation*}
a_p = (1 + 1/\varepsilon)^{p-1} \sigma_m^p,\quad b_p = 2^{p-1} ( \sigma_m^p + L^p).
\end{equation*}
\end{lemma}
\begin{proof}
\ref{lm:stat_prop_moment}: We have
\begin{equation*}
\bbE[\norm{\tilde{\ug}}^r] = \bbE[\norm{\tilde{\ug}}^p \norm{\tilde{\ug}}^{r-p}] \leq \lambda^{r-p}\bbE[\norm{\tilde{\ug}}^p].
\end{equation*}
Now, by the convexity of $t \mapsto t^p$ with $p \geq 1$, the Lipschitz continuity of $f$, and the assumption on the noise, it follows that
\begin{align*}
\bbE \norm{\tilde{\ug}}^p &\leq \bbE \norm{\bar{\ug}}^p = \bbE \norm{\bar{\ug} - \uu + \uu}^p \leq \bbE (\norm{\bar{\ug} - \uu} + \norm{\uu})^p \\
&= 2^p \bbE \Big(\frac{1}{2} \norm{\bar{\ug} - \uu} + \frac{1}{2} \norm{\uu}\Big)^p \\
&\leq 2^{p-1} \Big( \bbE \norm{\bar{\ug} - \uu}^p + \|\uu\|^p \Big) \leq 2^{p-1}\left(
\sigma_m^p
+ L^p\right) = b_p.
\end{align*}

\ref{lm:stat_prop_bias}: Let $\chi \coloneqq \bbI\big(\|\bar{\ug}\| > \lambda\big)$. 
Recall that by assumption $\lambda > L \geq \norm{\uu}$, so on the event $\|\bar{\ug}\| > \lambda$ we have
\begin{align*}
\|\bar{\ug}\| \leq \|\bar{\ug}-\uu\| + \|\uu\|\ \Rightarrow\ \|\bar{\ug} - \uu\| > \|\bar{\ug}\| - \|\uu\| > \lambda - L.
\end{align*}
Thus, setting $\chi' \coloneqq \bbI\big(\|\bar{\ug}- \uu\| > \lambda - L \big)$, we have $\chi \leq \chi'$. Moreover, by Markov's inequality,
\begin{align*}
\bbE\big[\chi' \big] = \bbP \big( \|\bar{\ug} - \uu \| \geq \lambda-L \big) \leq 
\frac{\sigma_m^p}{(\lambda-L)^p}
\end{align*}
Now note that by definition, $\tilde{\ug} = (1-\chi) \bar{\ug} + \chi (\lambda\norm{\bar{\ug}}^{-1} \bar{\ug}) = \bar{\ug} + \chi (\lambda \| \bar{\ug} \|^{-1} - 1) \bar{\ug}$. Thus,
\begin{align*}
\norm{\bbE[\tilde{\ug}] - \uu} &= \norm{\bbE[\tilde{\ug} - \bar{\ug}]} \leq \bbE \norm{\tilde{\ug} - \bar{\ug}} \\[0.7ex]
&= \bbE \norm{\chi (\lambda \norm{\bar{\ug}}^{-1} - 1) \bar{\ug}} 
= \bbE \big[\chi \abs{\lambda \norm{\bar{\ug}}^{-1} - 1} \norm{\bar{\ug}}\big] \\[1ex]
&= \bbE \big[ \chi(\norm{\bar{\ug}} - \lambda) \big] \leq \bbE \big[ \chi(\norm{\bar{\ug}} - \norm{\ug}) \big] \leq \bbE \big[\chi\norm{\bar{\ug} - \uu}\big] \\[0.7ex]
&\leq \bbE\big[ \chi' \norm{\bar{\ug} - \uu}\big] \leq \bbE \bigg[ \chi'\norm{\bar{\ug} - \uu}\Big(\frac{\norm{\bar{\ug}-\uu}}{\lambda-L}\Big)^{p-1} \bigg]\\
&\leq \bbE \norm{\bar{\ug} - \uu}^p\Big(\frac{1}{\lambda-L}\Big)^{p-1} \leq 
\frac{\sigma_m^p}{(\lambda-L)^{p-1}}.
\end{align*}
Finally, since the function $x \mapsto t/(t-L)$ is decreasing over the interval $\left[(1+\varepsilon)L,\infty\right[$ 
we have $(1+\varepsilon)/\varepsilon \geq \lambda/(\lambda-L)$ and hence $\lambda^{-1}(1+\varepsilon^{-1})\geq (\lambda-L)^{-1}$. The statement follows by recalling the definitions of $a_p$ and $b_p$.

\ref{lm:stat_prop_var}: By \ref{lm:stat_prop_moment} we have
\begin{align*}
\bbE \big[\norm{\tilde{\ug} - \bbE[\tilde{\ug}]}^2\big] &\leq \bbE \big[\norm{\tilde{\ug}}^2\big] \leq b_p \lambda^{2-p}.
\end{align*}
This concludes the proof.
\end{proof}

\subsection{Convergence in Expectation}

We prove \Cref{lm:useful_bounds} and \Cref{lm:recurrences}, restated below for reader's convenience.

\statistical*
\begin{proof}
\ref{eq:scalarp}:
Let $\xx \in \HH$.
By applying the tower rule of expectation, noting that $x_k-\xx$ is $\cF_k$-measurable, and applying \Cref{lm:stat_prop}\ref{lm:stat_prop_bias}, one obtains
\begin{align*}
\bbE [\scalarp{\tilde{\ug}_k-\ug_k,x_k-\xx}] &= \bbE\big[ \bbE_k[\scalarp{\tilde{\ug}_k-\ug_k,x_k-\xx}] \big]=
\bbE \big[ \scalarp{\bbE_k [\tilde{\ug_k}]-\ug_k,x_k-\xx} \big] \\
& \leq \bbE \big[ \|\bbE_k [\tilde{\ug_k}] - \ug_k \| \cdot \|x_k-\xx\| \big] \leq \bigg(1+\frac{1}{\varepsilon}\bigg)^{p-1} 
\sigma_m^p
\cdot \frac{\bbE \norm{x_k-\xx}}{\lambda_k^{p-1}}.
\end{align*}
\ref{eq:norm_squared}: We derive from \Cref{lm:stat_prop}\ref{lm:stat_prop_moment} that
\begin{align*}
\bbE \big[\norm{\tilde{\ug}_k}^2\big] &= \bbE \big[\bbE_k \big[\norm{\tilde{\ug}_k}^2\big]\big] = \bbE \big[\bbE_k \big[\norm{\tilde{\ug}_k}^p \norm{\tilde{\ug}_k}^{2-p}\big]\big] 
\leq \bbE \big[\bbE_k \big[\norm{\tilde{\ug}_k}^p \lambda_k^{2-p}\big]\big] \leq b_p \cdot \lambda_k^{2-p}.
\end{align*}
The statement follows.
\end{proof}
\recurrences*
\begin{proof}
\ref{lm:recurrences_ii}:
By summing the inequality $\beta_i + \alpha_{i+1}-\alpha_i \leq b_i \sqrt{\alpha_i}+c_i$ from $i=1$ to $i=k$, we obtain
\begin{equation}
\label{eq:20250415a}
    \sum_{i=1}^k \beta_i + \alpha_{k+1} \leq \alpha_1 + \sum_{i=1}^k c_i + \sum_{i=1}^k b_i \sqrt{\alpha_i}.
\end{equation}
Thus, setting $S_k = \alpha_1 + \sum_{i=1}^k c_i$, we have
\begin{equation*}
\forall k\in \N_0\colon\ \alpha_{k+1} \leq S_k + \sum_{i=1}^k b_i \sqrt{\alpha_i}.
\end{equation*}
This inequality is essentially the discrete version of the one occurring in the Bihari--LaSalle Lemma. So we can proceed similarly to the proof of such lemma. Set, for every $k \in \N_0$, $\gamma_{k+1} = S_k + \sum_{i=1}^k b_i \sqrt{\alpha_i}$. Then it is clear that $(\gamma_k)_{k \in \N}$ is increasing and that $\alpha_k\leq \gamma_k$. Therefore,
\begin{equation*}
    \forall\, k \in \N_0\colon\ \gamma_{k+1} = S_k + \sum_{i=1}^k b_i \sqrt{\alpha_i} \leq S_k + \sum_{i=1}^k b_i \sqrt{\gamma_i}\leq S_k + \bigg(\sum_{i=1}^k b_i\bigg) \sqrt{\gamma_{k+1}}.
\end{equation*}
Thus, $\sqrt{\gamma_{k+1}}$ is a positive solution of the quadratic equation
\begin{equation*}
    t^2 - \bigg( \sum_{i=1}^k b_i\bigg) t - S_k \leq 0
\end{equation*}
and hence
\begin{equation}
\label{eq:20250416a}
    \sqrt{\gamma_{k+1}} \leq \frac 1 2 \sum_{i=1}^k b_i+ \bigg( \Big(\frac1 2 \sum_{i=1}^k b_i\Big)^2 + S_k\bigg)^{1/2}.
\end{equation}
In the end \ref{lm:recurrences_ii} follows by noting that $\sqrt{\alpha_{k+1}}\leq \sqrt{\gamma_{k+1}}$, $\sqrt{a+b}\leq \sqrt{a}+ \sqrt{b}$,
and that
the right-hand side of \eqref{eq:20250416a} is increasing.

\ref{lm:recurrences_i}:
Setting $A_k = \max_{1\leq i\leq k}\sqrt{\alpha_i}$,
we derive from \ref{lm:recurrences_ii}, that
\begin{align*}
   \sum_{i=1}^k b_i \sqrt{\alpha_i}\leq  A_k\sum_{i=1}^k b_i 
    &\leq \frac{1}{2\varepsilon} A^2_k+  \frac{\varepsilon}{2}\bigg(\sum_{i=1}^k b_i\bigg)^2 \\
    &\leq \frac{1}{4\varepsilon \delta} \bigg(\sum_{i=1}^k b_i\bigg)^2+ \frac{\delta}{4\varepsilon} \bigg(\alpha_1 + \sum_{i=1}^k c_i \bigg) + \frac{\varepsilon}{2} \bigg(\sum_{i=1}^k b_i\bigg)^2
\end{align*}
In the end, recalling \eqref{eq:20250415a}, we have
\begin{equation*}
    \sum_{i=1}^k \beta_i + \alpha_{k+1} \leq \bigg(1+ \frac{\delta}{4\varepsilon}\bigg)\bigg(\alpha_1 + \sum_{i=1}^k c_i\bigg)+ \frac{1}{2}\bigg( \frac{1}{2\varepsilon \delta}+\varepsilon\bigg) \bigg(\sum_{i=1}^k b_i \bigg)^2
\end{equation*}
Now
\begin{equation*}
    \frac{1}{2\varepsilon \delta}+ \varepsilon = 2
    \ \Rightarrow\ \delta = \frac{1}{2\varepsilon (2-\varepsilon)}
\end{equation*}
and with such $\delta$ we have
\begin{equation*}
    \sum_{i=1}^k \beta_i + \alpha_{k+1} \leq \bigg(1+ \frac{1}{8\varepsilon^2(2-\varepsilon)}\bigg)\bigg(\alpha_1 + \sum_{i=1}^k c_i\bigg) + \bigg(\sum_{i=1}^k b_i \bigg)^2.
\end{equation*}
Optimizing $\varepsilon$, we get $\varepsilon =4/3$
and the statement follows.
\end{proof}
The next results is useful to prove the subsequent results.
\begin{lemma}
\label{lm:technical_lemma}
Let $k\in \N$. Then 
\begin{equation*}
\sum_{i = 1}^k \frac{1}{i (1+\log i)} \leq 1 + \log (1+\log k)
\end{equation*}
\end{lemma}
\begin{proof}
Since $t \mapsto [t (1 + \log t)]^{-1}$ is decreasing in $\R_{++}$, it holds
\begin{align*}
\sum_{i=1}^k \frac{1}{i (1 + \log i)} &= 1 + \sum_{i=2}^k \frac{1}{i (1 + \log i)} 
\leq 1 + \int_{1}^{k} \frac{1}{t (1+\log t)} dt \\
&= 1 + \int_{1}^{1+\log k} \frac{1}{u} du 
= 1 + \Bigg[ \log \abs{u} \Bigg]_{1}^{1+\log k} \leq 
1 + \log (1 + \log k).
\end{align*}
The statement follows.
\end{proof}
We now restate and prove \Cref{crl:avg_unbounded}.
\avgunbounded*
\begin{proof}
\ref{crl:avg_unbounded_at}. Plugging the parameters' definitions into the result of \Cref{thm:avg_unbounded}, leads to 
\begin{align*}
\bbE[f(\bar{x}_k) - f_*] \leq \frac{(1 + \log k)^{1/p}}{\gamma k^{(p-1)/p}} \Bigg[&\norm{x_{1}-\xx_*}^2 + 2 a_p^2 \bigg(\frac{\gamma}{\lambda^{p-1}}\sum_{i=1}^k \frac{1}{i (1 + \log i)} \bigg)^2 \\
&\qquad+ b_p \max\{\Le,\lambda\}^{(2-p)} \gamma^2 \sum_{i=1}^k \frac{1}{i (1 + \log i)} \Bigg].
\end{align*}
Now, by \Cref{lm:technical_lemma} and setting for the sake of brevity $\alpha_k = (1+\log(1+\log k))$, we have
\begin{align*}
\bbE[f(\bar{x}_k)  - f_*] \leq \frac{(1 + \log k)^{1/p}}{k^{(p-1)/p}} \Bigg[&\frac{\norm{x_{1}-\xx_*}^2}{\gamma} + \gamma\Bigg( 2  \bigg(\frac{a_p}{\lambda^{p-1}}\bigg)^2 \alpha_k^2 \\
&\hspace{10ex}+ b_p \max\{\Le,\lambda\}^{(2-p)} \alpha_k \Bigg)  \Bigg].
\end{align*}
The statement follows.

\ref{crl:avg_unbounded_fh}. Applying the parameters' definitions to the result of \Cref{thm:avg_unbounded}, leads to
\begin{align*}
\bbE[f(\bar{x}_k) &- f_*] \leq \frac{1}{\gamma k^{(p-1)/p}} \Bigg(\norm{x_{1}-\xx_*}^2 +  2 \gamma^2 \bigg( \frac{a_p}{\lambda^{p-1}} \bigg)^2 + b_p \max\{\Le,\lambda\}^{(2-p)} \gamma^2 \Bigg),
\end{align*}
and bringing $1/\gamma$ in the parentheses, (\ref{eq:avg_unbounded_fh})  follows.     
\end{proof}
\subsection{Convergence in high probability}
%
%
%
%
%
Here we restate and prove \Cref{crl:avg_unbounded_hp}.
\crlavgunboundedhp*
\begin{proof}
\ref{crl:avg_unbounded_hp_at}. Set $(\gamma_i)_{i \in \bbN}$ and  $(\lambda_i)_{i \in \bbN}$ as in \ref{AT}. Then we can control $A_k, B_k, C_k, D_k,$ and $E_k$ as follows. We set again, for the sake of brevity, $\eta_i = i(1+\log i)$ and $\alpha_k = 1+ \log(1+\log k)$.
\begin{itemize}
\item For $A_k$ and $B_k$, it follows from the last line of the proof of  in \Cref{crl:avg_unbounded}, that
\begin{align}
\label{eq:A_k_at}
A_k &\leq \frac{\gamma a_p}{\lambda^{p-1}}\alpha_k = \gamma\frac{v_p}{4} \alpha_k, \\
\label{eq:B_k_at}
B_k &\leq b_p  \gamma^2 \max\{\Le, \lambda\}^{2-p} \alpha_k = \gamma^2 w_p \alpha_k.
\end{align}
\item For $C_k$ and $D_k$ we have
\begin{align}
\label{eq:C_k_at}
C_k &= \max_{i \in [k]} \gamma_i \lambda_i = \max_{i \in [k]} \frac{\gamma}{\eta_i^{1/p}} \max\{\Le, \lambda \eta_i^{1/p}\} \leq \gamma  \max\{\Le, \lambda\}, \\
\label{eq:D_k_at}
D_k &=  \max_{i \in [k]} \gamma_i^2 \lambda_i^2 =  C_k^2 \leq
\gamma^2 \max\{\Le, \lambda\}^2.
\end{align}
\item Finally, for $E_k$ we have
\begin{align}
\label{eq:E_k_at}
E_k &= b_p \sum_{i=1}^{k} \gamma_i^4 \lambda_i^{4-p} = b_p \sum_{i=1}^{k}  \frac{\gamma^4}{\eta_i^{4/p}} \max\{\Le, \lambda \eta_i^{1/p}\}^{4-p} \leq 
\gamma^4 w_p \max\{\Le, \lambda\}^2 \alpha_k.
\end{align}
\end{itemize}
Putting \eqref{eq:A_k_at}-\eqref{eq:E_k_at} together in \eqref{eq:avg_unbounded_hp}, yields to
\begin{align*}
f(\bar{x}_k) -f_* &\leq \frac{(1+\log k)^{1/p}}{\gamma k^{(p-1)/p}} \Bigg[ \norm{x_{1}-x_*} +  
\gamma \frac{v_p}{4}\alpha_k
+ \gamma^2 w_p \alpha_k\nonumber\\
&\hspace{16.2ex}+\frac{8}{3}\gamma  \max\{\Le, \lambda\}\big(1 +\gamma \max\{\Le, \lambda\} \big) \cdot \log \Big(\frac{2}{\delta}\Big) \\
&\hspace{16.2ex}+ \gamma\sqrt{ w_p \alpha_k} 
\Big( 1+ \gamma\max\{\Le, \lambda\}\Big) \cdot \sqrt{\log \bigg(\frac{2}{\delta}\bigg)} \Bigg]^2 \\
&\leq  \frac{(1+\log k)^{1/p}}{\gamma k^{(p-1)/p}} 4 \Bigg[ \norm{x_{1}-x_*}^2 + \gamma^2\bigg( \frac{v_p}{4} + \gamma w_p\bigg)^2\alpha_k^2  \nonumber\\
&\hspace{18.2ex}+\frac{64}{9}\gamma^2  \max\{\Le, \lambda\}^2\big(1 +\gamma \max\{\Le, \lambda\} \big)^2 \cdot \log^2 \Big(\frac{2}{\delta}\Big) \\
&\hspace{18.2ex}+ w_p \gamma^2\alpha_k \big(1+ \gamma \max\{\Le, \lambda\} \big)^2 \cdot \log \bigg(\frac{2}{\delta}\bigg) \Bigg]\\
&=  \frac{(1+\log k)^{1/p}}{k^{(p-1)/p}}   \Bigg[ \frac{4}{\gamma}\norm{x_{1}-x_*}^2 + \gamma\bigg( \frac{v_p}{2} + 2\gamma w_p\bigg)^2\alpha_k^2  \nonumber\\
&\hspace{16.2ex}+32\gamma  \max\{\Le, \lambda\}^2\big(1 +\gamma \max\{\Le, \lambda\} \big)^2 \cdot \log^2 \Big(\frac{2}{\delta}\Big) \\
&\hspace{16.2ex}+ 4 w_p \gamma\alpha_k \big(1+ \gamma \max\{\Le, \lambda\} \big)^2 \cdot \log \bigg(\frac{2}{\delta}\bigg) \Bigg]
\end{align*}
Finally, substituting $\gamma = (\log(2/\delta))^{-1}$, since $\gamma\leq 1$ (being $\delta\leq 2/e$),
one obtains
\begin{align*}
f(\bar{x}_k) -f_* &\leq \frac{(1+\log k)^{1/p}}{k^{(p-1)/p}}   \Bigg[ 4\norm{x_{1}-x_*}^2 \log\Big(\frac{2}{\delta}\Big) + \bigg( \frac{v_p}{2} + 2 w_p\bigg)^2\alpha_k^2  \nonumber\\
&\hspace{18.2ex}+32  \max\{\Le, \lambda\}^2\big(1 + \max\{\Le, \lambda\} \big)^2 \cdot \log \Big(\frac{2}{\delta}\Big) \\
&\hspace{18.2ex}+ 4 w_p \alpha_k \big(1+ \max\{\Le, \lambda\} \big)^2 \Bigg]\\
&= \frac{(1+\log k)^{1/p}}{k^{(p-1)/p}}   \Bigg[  \bigg( \frac{v_p}{2} + 2 w_p\bigg)^2\alpha_k^2  +4 w_p \alpha_k \big(1+ \max\{\Le, \lambda\} \big)^2\nonumber\\
&\hspace{18.2ex}+\Big(32  \max\{\Le, \lambda\}^2\big(1 + \max\{\Le, \lambda\} \big)^2 + 4\norm{x_{1}-x_*}^2\Big) \cdot \log \Big(\frac{2}{\delta}\Big)  \Bigg]
\end{align*}
%
%
The statement follows.

\ref{crl:avg_unbounded_hp_fh}. Set $(\gamma_i)_{1 \leq i \leq k}$ and  $(\lambda_i)_{1 \leq i \leq k}$ as in \ref{FH}. Then we can control $A_k, B_k, C_k, D_k,$ and $E_k$ as follows.
\begin{itemize}
\item For $A_k$ and $B_k$, it follows from the last line of the proof of \Cref{crl:avg_unbounded}, that
\begin{align}
\label{eq:A_k_fh}
A_k &\leq a_p \frac{\gamma}{\lambda^{p-1}} =\gamma \frac{v_p}{4}, \\
\label{eq:B_k_fh}
B_k &\leq \gamma^2 b_p \max\{\Le, \lambda\}^{2-p}  = \gamma^2 w_p;
\end{align}
\item For $C_k$ and $D_k$ notice that
\begin{align}
\label{eq:C_k_fh}
C_k &= \max_{i \in [k]} \gamma_i \lambda_i = \frac{\gamma}{k^{1/p}} \max\{\Le, \lambda k^{1/p}\} \leq \gamma  \max\{\Le, \lambda\}, \\
\label{eq:D_k_fh}
D_k &=  \max_{i \in [k]} \gamma_i^2 \lambda_i^2 = \bigg(\frac{\gamma}{k^{1/p}} \max\{\Le, \lambda k^{1/p}\}\bigg)^2 \leq  \gamma^2 \max\{\Le, \lambda\}^2;
\end{align}
\item Finally, for $E_k$ notice that
\begin{align}
\label{eq:E_k_fh}
E_k &= b_p \sum_{i=1}^{k} \gamma_i^4 \lambda_i^{4-p}  \leq b_p \gamma^4 \max\{\Le, \lambda\}^{4-p} = \gamma^4 w_p \max\{\Le, \lambda\}.
\end{align}
\end{itemize}
We note that the bounds in \eqref{eq:A_k_fh}-\eqref{eq:E_k_fh}
are exactly the same of those in \eqref{eq:A_k_at}-\eqref{eq:E_k_at} except for the presence of $\alpha_k$. So we can proceed similarly by just replacing $\alpha_k$ with $1$ in the derivation given in \ref{crl:avg_unbounded_hp_at}. Also, since $\gamma_i$ is now constant along the $k$ iterates, we will have $\big((1+\log k)^{1/p}\big)/k^{(p-1)/p}$ replaced by $1/k^{(p-1)/p}$ in front of the bound.

\end{proof}

\section{Technical lemmas for the analysis of the last iterate}
\label{app:B}
\subsection{Finite-horizon setting}
%
We provide the proof of the following result taken from \cite{Jain2021}.

\keyepochlenght*
%
%
\begin{proof}
Noting that for every $x \geq 0$, it holds $2 \lceil x \rceil - 1 \leq \lceil 2x \rceil \leq 2 \lceil x \rceil$ and recalling the definition of $k_j$, we have
\begin{align*}
k_{j+1}-k_j &= \lceil k/2^{j} \rceil - \lceil k/2^{j+1} \rceil = \lceil 2 \cdot (k/2^{j+1}) \rceil - \lceil 2\cdot (k/2^{j+2}) \rceil \leq 2 \lceil k/2^{j+1} \rceil - 2 \lceil k/2^{j+2} \rceil + 1 \\
&= 2(\lceil k/2^{j+1} \rceil - \lceil k/2^{j+2} \rceil) + 1 = 2(k_{j+2}-k_{j+1} ) + 1 \\
&\leq 4 (k_{j+2}-k_{j+1} ),
\end{align*}
where in the last inequality we used the fact $k_{j+2}-k_{j+1}$
is an integer  $\geq 1$.
\end{proof}
%
%
%

\subsection{Anytime setting}
For completeness, we give a proof of the results of \Cref{lm:key_inquelity_at},
which appeared in \cite{Lin2018}.
%
\keyinequalityat*
\begin{proof}
\ref{lm:key_inquelity_at_i}:
For every $j \in [k]$ define $m_j = \frac{1}{j} \sum_{i=k-j+1}^k \beta_i$ (which is the arithmetic mean of the last $j$ terms in the sequence $(\beta_i)_{1 \leq i \leq k}$). Then, for $j<k$, we have
\begin{align*}
m_j - m_{j+1} &= \frac{1}{j} \sum_{i=k-j+1}^k \beta_i - \frac{1}{j+1} \sum_{i=k-j}^k \beta_i\\
& = \frac{1}{j(j+1)}\bigg[(j+1) \sum_{i=k-j+1}^k \beta_i- j  \sum_{i=k-j}^k \beta_i \bigg]\\
& = \frac{1}{j(j+1)}\bigg[\sum_{i=k-j+1}^k \beta_i+j \bigg(\sum_{i=k-j+1}^k \beta_i-   \sum_{i=k-j}^k \beta_i\bigg) \bigg]\\
& = \frac{1}{j(j+1)}\bigg[\sum_{i=k-j+1}^k \beta_i-j \beta_{k-j} \bigg]\\
& = \frac{1}{j(j+1)} \sum_{i=k-j+1}^k (\beta_i- \beta_{k-j}).
\end{align*}
Therefore,
\begin{equation*}
\sum_{j=1}^{k-1} \frac{1}{j(j+1)} \sum_{i=k-j+1}^k (\beta_i- \beta_{k-j}) = \sum_{j=1}^{k-1} (m_j- m_{j+1})
= m_1 - m_k  = \beta_k - \frac{1}{k}\sum_{i=1}^k \beta_i
\end{equation*}
and the statement follows.

\ref{lm:key_inquelity_at_iii}:
We start by noting that
\begin{equation}
\label{eq:20250403b}
\sum_{j=1}^{k-1} \frac{1}{j(j+1)}\sum_{i=k-j}^k \beta_i
= \sum_{i=1}^{k-1} \frac{\beta_i}{k-i} - \frac 1 k \sum_{i=1}^k \beta_i + \beta_k.
\end{equation}
Indeed, define $\delta_{i,j} = 1$, if $i\geq k-j$ and $0$, otherwise. Then
exchanging the order in the sum, we have
\begin{align*}
\sum_{j=1}^{k-1}\frac{1}{j(j+1)} \sum_{i=k-j}^k \beta_i 
&=\sum_{j=1}^{k-1}\frac{1}{j(j+1)} \sum_{i=k-j}^{k-1} \beta_i + \sum_{j=1}^{k-1}\frac{1}{j(j+1)} \beta_k\\ 
&=\sum_{j=1}^{k-1} \sum_{i=1}^{k-1} \frac{1}{j(j+1)} \beta_i\delta_{i,j} + \sum_{j=1}^{k-1}\frac{1}{j(j+1)} \beta_k\\ 
&= \sum_{i=1}^{k-1}\sum_{j=1}^{k-1}\frac{1}{j(j+1)} \beta_i  \delta_{i,j} + \sum_{j=1}^{k-1}\frac{1}{j(j+1)} \beta_k\\ 
&=\sum_{i=1}^{k-1}\sum_{j=k-i}^{k-1} \frac{1}{j(j+1)} \beta_i+ \sum_{j=1}^{k-1}\frac{1}{j(j+1)} \beta_k\\
&=\sum_{i=1}^{k-1}\bigg( \frac{1}{k-i}- \frac{1}{k}\bigg) \beta_i+ \bigg( 1- \frac{1}{k}\bigg) \beta_k\\
&= \sum_{i=1}^{k-1} \frac{\beta_i}{k-i} - \frac{1}{k}\sum_{i=1}^k \beta_i + \beta_k.
\end{align*}
Equation~\eqref{eq:20250403b} follows. Now, set $\eta_i = i \beta_i$
and note that
\begin{align*}
\sum_{i=1}^{k-1} \frac{\beta_i}{(k-i)} 
& = \sum_{i=1}^{k-1} \frac{\eta_i}{i(k-i)}\\
&= \frac{1}{k} \sum_{i=1}^{k-1} \Bigg( \frac{1}{i} + \frac{1}{k-i} \Bigg) \eta_i = \frac{1}{k} \sum_{i=1}^{k-1} \frac{\eta_i}{i} + \frac{1}{k} \sum_{i=1}^{k-1} \frac{\eta_i}{k-i}.
\end{align*}
Since $k-i\leq i$ iff $i \geq \lceil k/2 \rceil$, 
we have
\begin{align*}
 \sum_{i=1}^{k-1}  \frac{\eta_i}{k-i} &=  \sum_{i=1}^{\lceil k/2 \rceil-1} \frac{\eta_i}{k-i} 
+ \sum_{i=\lceil k/2 \rceil}^{k-1} \frac{\eta_i}{k-i} \overset{(*)}{\leq}  \sum_{i=1}^{\lceil k/2 \rceil-1} \frac{\eta_i}{i} +  \sum_{i=\lceil k/2 \rceil}^{k-1} \frac{\eta_{k-i}}{k-i} \leq 2\sum_{i=1}^{k-1} \frac{\eta_i}{i}
\end{align*}
where in inequality $(*)$ we used that the $\eta_i$'s are decreasing.
Hence
\begin{equation*}
\sum_{i=1}^{k-1} \frac{\beta_i}{(k-i)}  \leq \frac{3}{k}\sum_{i=1}^{k-1} \beta_i = 
\frac{3}{k}\sum_{i=1}^{k} \beta_i - \frac{3}{k} \beta_k.
\end{equation*}
Finally, combining the above inequality with \eqref{eq:20250403b}, the statement follows.
\end{proof}
%

\section{Implementation of the clipped SsGM with kernels}
\label{sec:appC}
In this section we show that the assumptions \ref{H1}-\ref{H4} are satisfied for problem
\eqref{eq:statistical_learning} and provide the related details concerning the implementation of Algorithm \ref{algorithm:clippedSsGM}.

\paragraph{Checking assumptions}
\begin{itemize}
\item Hypothesis \ref{H1} is satisfied. 
Assumption \ref{A4} implies that 
\begin{align*}
\ell(\scalarp{\xx, \phi(Z)}, Y) &\leq \abs{\ell(\scalarp{\xx, \phi(Z)}, Y) - \ell(0,Y)} + \ell(0,Y)\\
&\leq \mathcal{L} \abs{\scalarp{\xx, \phi(Z)}} + \ell(0,Y)
\leq \mathcal{L} \norm{\xx}\norm{\phi(Z)} + \ell(0,Y)
\end{align*}
and hence, since $\bbE [\norm{\phi(Z)}] < \infty$ and $\bbE[\ell(0,Y)]<+\infty$,
due to assumptions \ref{A3} and \ref{A4}, we have that $R(\xx)<+\infty$ for every $\xx \in \HH$.
Moreover, for every $\xx_1,\xx_2\in \HH$
\begin{equation*}
\abs{\ell(\scalarp{\xx_1, \phi(Z)}, Y) - \ell(\scalarp{\xx_2, \phi(Z)}, Y)}
\leq \mathcal{L} \abs{\scalarp{\xx_1-\xx_2, \phi(Z)}} \leq \mathcal{L}\norm{\xx_1- \xx_2}
\norm{\phi(Z)}
\end{equation*}
 and taking the expectation we have that
 $R(\xx)$ is Lipschitz continuous with constant $L = \cL \cdot \bbE [\norm{\phi(Z)}]$.
Finally, by the convexity of $\ell(\cdot, y)$, 
if $\xx_1, \xx_2 \in \HH$ and $\alpha\in [0,1]$, we have
\begin{equation*}
\ell(\scalarp{(1-\alpha)\xx_1+\alpha \xx_2, \phi(Z)}, Y) \leq (1-\alpha)\ell(\scalarp{\xx_1, \phi(Z)}, Y) +\alpha \ell(\scalarp{\xx_2, \phi(Z)}, Y). 
\end{equation*}
Thus, taking the expectation the convexity of $R$ follows. 
\item  Hypothesis \ref{H2} holds with $\XX = \HH$.
\item  $\hat{\ug}(\xx, (Z, Y))$ satisfies  \ref{H3} with $\xi=(Z,Y)$. Indeed, 
defining $\varphi(\xx, \xi) = \ell(\scalarp{\xx,\phi(Z)}, Y)$, we have $\ell^\prime (\scalarp{\xx,\phi(Z)}, Y) \phi(Z) \in \partial\varphi (\cdot, \xi)(\xx)$ and hence
\begin{equation*}
\forall\, \xx^\prime \in \HH\colon  \scalarp{\xx^\prime-\xx,\hat{\ug}(\xx, (Z, Y))}+ \varphi(\xx,\xi) \leq \varphi(\xx^\prime,\xi).
\end{equation*}
Thus, taking the expectation, we get $\uu(\xx)=\EE[\hat{\ug}(\xx, (Z, Y))] \in \partial R(\xx)$
\item Assumption \ref{A3} implies ${\rm\ref{H4}}$ with $\sigma^p = (2\nu \cL)^p$. Indeed,
\begin{align*}
\nonumber
\bbE \norm{\hat{\ug}(\xx,(Z,Y)) - \uu(\xx)}^p &\leq 2^{p-1} \big(\bbE [\norm{\hat{\ug}(\xx,(Z,Y))}^p] +\norm{\uu(\xx)}^{p}\big)\\
&\leq 2^{p-1}\big( \bbE [\norm{\hat{\ug}(\xx,(Z,Y))}^p] + \norm{\bbE[\hat{\ug}(\xx,(Z,Y))]}^p\big) \\
&\leq 2^{p-1}\big( 2\bbE [| \ell^\prime(\scalarp{\xx, \phi(Z), Y)}|^{p} \norm{\phi(Z)}^{p}] \big) \\
&\leq (2\cL)^{p} \bbE \norm{\phi(Z)}^{p} \leq (2\nu \cL)^p < \infty.
\end{align*}
\end{itemize}

\paragraph{Implementation details}

Since $\XX=\HH$, the projection $P_{\XX}$ is the identity mapping. Thus, for each $k \geq 1$, the main update in Algorithm \ref{algorithm:clippedSsGM} is defined by the following
\begin{equation}
x_{k+1} = x_k - \gamma_k \tilde{\ug}_k,
\label{eq:clipped-SsGM}
\end{equation}
where $\tilde{\ug}_k = \textsc{CLIP}(\bar{\ug}_k,\lambda_k)$, $\bar{\ug}_k = \frac{1}{m} \sum_{j=1}^m \hat{\ug}(x_k, (Z_j^k, Y_j^k))$, $\hat{\ug}(\xx, (Z, Y)) = \ell^\prime(\langle \xx, \phi(Z) \rangle, Y) \phi(Z)$ (recall that, for all $(t,y) \in \bbR \times \cY$, $\ell^\prime(t,y)$ is a subgradient of $\ell(\cdot, y)$ at $t$, that is, $\ell^\prime(t,y) \in \partial\ell(t, y)$). Since $\HH$ may be infinite dimensional, computing $\tilde{\ug}_k$ with a straightforward application of its definition may be problematic.

In order to solve the aforementioned issue, the algorithm will keep an implicit representation of the iterates $x_k$'s in terms of the kernels. By using the definitions of clipping and $\hat{\ug}(x, (Z, Y))$ it holds that
\begin{equation}
\label{eq:kernel_clip}
\tilde{\ug}_{k} 
= 
\rho_k
\cdot \bar{\ug}_k,
\qquad\text{with}\quad
\bar{\ug}_k= \frac{1}{m} \sum_{j=1}^m \hat{\ug}(x_k, (Z_j^k, Y_j^k)) = \frac{1}{m} \sum_{j=1}^m \alpha_j^k \phi(Z_j^k),
\end{equation}
where we set $\rho_k 
= \max \left\{\|\bar{\ug}_k\|/\lambda_k, 1 \right\}^{-1}$ 
and $\alpha_j^k =  \ell^\prime(\langle x_k, Z_j^k \rangle , Y_j^k)$. From equation \eqref{eq:kernel_clip} we have
\begin{equation}
\|\bar{\ug}_k\|^2 = \frac{1}{m^2} \sum_{j,j'=1}^m \alpha_j^k \alpha_{j'}^k K(Z_j^k, Z_{j'}^k).
\end{equation}
The above equation allows for the computation of $\rho_k$ once the $\alpha_j^k$s are known. Furthermore, combing  \Cref{eq:kernel_clip} and \Cref{eq:clipped-SsGM} we obtain that
\begin{align}
\label{eq:kernel_wk}
x_{k+1} &= x_k - \gamma_k \tilde{\ug}_k = x_k - \gamma_k \frac{\rho_k}{m} \sum_{j=1}^m \alpha_j^k\phi(Z_j^k) \nonumber\\
&= x_k + \sum_{j=1}^m \left(- \frac{\gamma_k\rho_k}{m} \right) \alpha_j^k\phi(Z_j^k) = \sum_{i=0}^{k} \sum_{j=1}^m a_{ij}^k \phi(Z_j^i) \:,
\end{align}
where we set $a^0_{i,j}=0$ and
\begin{equation}
a_{ij}^k = \begin{cases}
a_{ij}^{k-1} & \text{ if } i \leq k-1,\\[1ex]
-\dfrac{\gamma_k\rho_k}{m} \alpha_j^k & \text{ if } i = k.
\end{cases}    
\end{equation}
%
%
Replacing the expression of $x_k$ from \eqref{eq:kernel_wk} in the definition of $\alpha_j^k$ leads to
\begin{align}
\alpha_j^k = \ell'(\langle x_k, \phi(Z_j^k)\rangle, Y_j^k) = \ell^\prime \left(\sum_{i=0}^{k-1} \sum_{j'=1}^m a_{ij'}^{k-1} K(Z_{j'}^i, X_j^k), Y_j^k\right),
\end{align}
which can be computed directly using the kernels.

\begin{algorithm}[t!]
\caption{Kernel Clipped Stochastic subGradient Method}
Given the step-sizes $(\gamma_k)_{k \in \N} \in \bbR_{++}^\N$, the clipping levels $(\lambda_k)_{k \in \N} \in \bbR_{++}^\N$, the batch size $m \in \N$, $m\geq 1$, do the following.
\begin{equation}
\label{eq:loop}
\begin{array}{l}
\nonumber
\textsc{Initialization}: a_{i,j}^0=0\\[1ex]    
\text{for}\;k=1,\ldots\\[1ex]
\left\lfloor
\begin{array}{l}
\text{draw } (Z_{j}^k, Y_j^k)_{1\leq j \leq m} \ m \text{ independent copies of } (Z,Y),\\[1ex]
\text{set } \alpha_j^k =  \ell^\prime \left( \sum_{i=0}^{k-1}\sum_{j'=1}^{m} a_{ij'}^{k-1} K(Z_{j'}^i, Z_j^k), Y_j^k \right), \text{ for each } 1 \leq j \leq m,\\[1ex]
\text{compute}\\[1ex]
\left\lfloor
\begin{array}{l}
\|\bar{\ug}_k\|^2 = \frac{1}{m^2} \sum_{j,j' = 1}^{m} \alpha_j^k \alpha_{j'}^k K(Z_j^k, Z_{j'}^k),\\[1ex]
\rho_k 
= \max\left\{\frac{\|\bar{\ug}_k\|}{\lambda_k}, 1\right\}^{-1},\\[1ex]
a_{ij}^k = 
\begin{cases}
a_{ij}^{k-1}, & \text{ if } 0 < i \leq k-1, \\
-\frac{\gamma_k\rho_k}{m} \alpha_{j}^k & \text{ otherwise}.
\end{cases}
\end{array}
\right.
\end{array}
\right.
\end{array}
\end{equation}
From the sequence $(x_k)_{k \in \N}$ one defines also, for every $k \in \N$,
$\displaystyle\bar{x}_k = \frac{1}{k} \sum_{i=1}^k x_i$.
\label{algo:kernelSsGM}
\end{algorithm}

\paragraph{Algorithm.} The algorithm keeps an implicit representation for the iterates $x_k$ which are computed as in \eqref{eq:kernel_wk}. This is enough to make predictions on new points as we are going to show in the following. We notice that, at any time $k$, it is possible to make a prediction for an instance $Z$ using the $k$-th iterate as follows
\begin{align}
\label{eq:kernel_pred}
\langle x_{k+1}, \phi(Z) \rangle &= \langle x_{k}, \phi(Z) \rangle - \frac{\gamma_k\rho_k}{m} \sum_{j=1}^m \alpha_j^k K(Z_j^k, Z).
\end{align}
This prediction requires nothing but the prediction made with the previous iterate, $\rho_k, \alpha_j^k$s and the values of the kernels $K(Z_j^k,Z)$. It is easy to observe that, by recursion, there is no need to have an explicit expression for the $x_k$; instead it is enough to update, along the iterations, only $\rho_k, a_{ij}^k$ and $\alpha_j^k$. The full procedure is given in Algorithm \ref{algo:kernelSsGM}. Then, by considering $\bar{x}_k$ as defined in Algorithm \ref{algorithm:clippedSsGM}, the prediction  can be computed as
\begin{align}
\label{eq:kernel_w_pred}
\langle \bar{x}_{k+1}, \phi(Z) \rangle &= \frac{1}{k+1} \left((k \langle \bar{x}_{k}, \phi(Z) \rangle + \langle x_{k+1}, \phi(Z) \rangle \right).
\end{align}

\begin{remark}
Notice that the prediction made with the $k+1$-th iterate in \Cref{eq:kernel_pred} can be computed recursively from the prediction made by the $k$-th iterate. Similarly, the prediction made by the $k+1$-th average in \Cref{eq:kernel_w_pred} can be computed by the prediction made by the previous weighted average. Both there recursion allows for significant computational saving, as at each step it only necessary to compute the kernels among $X$ and the instances of the $k$-th batch.
\end{remark}

\end{appendices}

\end{document}